\documentclass[reqno]{amsart}

\makeatletter
\renewcommand\part{%
	\if@noskipsec \leavevmode \fi
	\par
	\addvspace{4ex}%
	\@afterindentfalse
	\secdef\@part\@spart}

\def\@part[#1]#2{%
	\ifnum \c@secnumdepth >\m@ne
	\refstepcounter{part}%
	\addcontentsline{toc}{part}{\thepart\hspace{1em}#1}%
	\else
	\addcontentsline{toc}{part}{#1}%
	\fi
	{\parindent \z@ \raggedright
		\interlinepenalty \@M
		\normalfont
		\ifnum \c@secnumdepth >\m@ne
		\Large\bfseries \partname\nobreakspace\thepart
		\par\nobreak
		\fi
		\huge \bfseries #2%
		\par}%
	\nobreak
	\vskip 3ex
	\@afterheading}
\def\@spart#1{%
	{\parindent \z@ \raggedright
		\interlinepenalty \@M
		\normalfont
		\huge \bfseries #1\par}%
	\nobreak
	\vskip 3ex
	\@afterheading}
\makeatother

\usepackage{color}
\usepackage[dvipsnames]{xcolor}

\usepackage{ifpdf}
\ifpdf 
    \usepackage[pdftex]{graphicx}   
    \usepackage[pdftex,     
            plainpages=false,   
            breaklinks=true,    
            colorlinks=true,
            linkcolor=red,
            citecolor=green,
            pdftitle=My Document
            pdfauthor=My Good Self
           ]{hyperref} 
\else 
    \usepackage{graphicx}       
    \usepackage{hyperref}       
\fi 

\usepackage{subfig}
\usepackage{glossaries}
\usepackage{glossary-mcols}

\usepackage{aurical}
\usepackage{amsfonts,amsmath}	
\usepackage{amssymb}
\usepackage{verbatim}
\usepackage{amsopn}
\usepackage[english]{babel}
\usepackage{amsthm}
\usepackage{enumerate}
\usepackage{mathrsfs}	
\usepackage{enumitem}
\usepackage{mathtools}
\usepackage{esint}
\usepackage{bbm}
\usepackage{caption}
\usepackage{marginnote}
\usepackage[marginparwidth=2cm]{geometry}
\usepackage{comment}

\date{\today}

\theoremstyle{definition} \newtheorem{definition}{Definition}[section]
\theoremstyle{definition} \newtheorem{remark}[definition]{Remark}
\theoremstyle{plain} \newtheorem{lemma}[definition]{Lemma}
\theoremstyle{plain} \newtheorem{proposition}[definition]{Proposition}
\theoremstyle{plain} \newtheorem{theorem}[definition]{Theorem}
\theoremstyle{plain} \newtheorem{corollary}[definition]{Corollary}
\theoremstyle{definition} \newtheorem{example}[definition]{Example}
\theoremstyle{plain} 
\theoremstyle{definition} \newtheorem{assumption}[definition]{Assumption}

\DeclareMathOperator{\dive}{div}

\DeclareMathOperator{\supp}{supp}

\DeclareMathOperator{\clos}{clos}

\DeclareMathOperator{\Lip}{Lip}

\DeclareMathOperator{\Dom}{Dom}
\DeclareMathOperator{\Graph}{Graph}

\DeclareMathOperator{\diam}{diam}

\newcommand{\R}{\mathbb{R}}
\newcommand{\Q}{\mathbb{Q}}
\newcommand{\N}{\mathbb{N}}
\newcommand{\Z}{\mathbb{Z}}
\newcommand{\C}{\mathbb{C}}

\newcommand{\loc}{\text{\rm loc}}

\newcommand{\Id}{\mathrm{id}}
\newcommand{\ind}{1\!\!\mathrm{I}}

\newcommand{\adm}{\text{\textsf{Adm}}}

\makeglossaries

\newglossaryentry{adm}{name={\ensuremath{\mathrm{Adm(\mu)}}},description={$\mu$-admissible transference plans $\pi \in \mathcal P(\Gamma)$}}
\newglossaryentry{An}{name={\ensuremath{A(n)}},description={set of curves made of up to pieces of $n$-different curves}}
\newglossaryentry{Api}{name={\ensuremath{A_\pi}},description={set of curves where the test plan $\pi$ is concentrated}}
\newglossaryentry{barApi}{name={\ensuremath{\bar A_\pi}},description={set of curves where the test plan $\pi$ is concentrated generating the minimal equivalence relation}}
\newglossaryentry{Ball}{name={\ensuremath{B^X_r(x)}},description={ball of radius $r$ centered in $x \in X$}}
\newglossaryentry{BallSet}{name={\ensuremath{B^X_r(A)}},description={$r$-neighborhood of the set $A \subset X$}}
\newglossaryentry{Borelsig}{name={\ensuremath{\mathcal B(X)}},description={Borel $\sigma$-algebra of the topological space $X$}}
\newglossaryentry{calApi}{name={\ensuremath{\mathcal A(\pi)}},description={increasing sequences of compact sets so that $\pi$ is concentrated on theri union}}
\newglossaryentry{calA}{name={\ensuremath{\mathcal A}},description={graph of the function $\pi \mapsto \mathcal A(\pi)$}}
\newglossaryentry{calCn}{name={\ensuremath{\mathcal C(n)}},description={subset of $A(n)$ made of connected curves}}
\newglossaryentry{cardA}{name={\ensuremath{\sharp A}},description={cardinality of the set $A$}}
\newglossaryentry{Cconst}{name={\ensuremath{\mathrm{const.}}},description={generic constant}}
\newglossaryentry{Ccompr}{name={\ensuremath{C}},description={compressibility constant of a test plan $\pi$}}
\newglossaryentry{ddistance}{name={\ensuremath{d}},description={distance on the metric space $X$}}
\newglossaryentry{df}{name={\ensuremath{Df}},description={cotangent vector field corresponding to the function $f \in S^1$}}
\newglossaryentry{uppegr}{name={\ensuremath{|Df|}},description={upper gradient of a function $f$}}
\newglossaryentry{Hausd}{name={\ensuremath{d_H(K_1,K_2)}},description={Hausdorff distance on $\mathcal K(X)$}}
\newglossaryentry{Equiv}{name={\ensuremath{E}},description={equivalence relation}}
\newglossaryentry{equivcla}{name={\ensuremath{Ex}},description={equivalence class containing $x$}}
\newglossaryentry{EApi}{name={\ensuremath{E^{A_\pi}}},description={equivalence relation generated by the curves in $A_\pi$}}
\newglossaryentry{barE}{name={\ensuremath{\bar E}},description={equivalence relation generated by countably many saturated sets}}
\newglossaryentry{Epi}{name={\ensuremath{E^\pi}},description={minimal equivalence relation generated by $\pi$}}
\newglossaryentry{evaluation}{name={\ensuremath{e(t)}},description={evaluation map $\Gamma \ni \gamma \mapsto \gamma(t) \in X$}}
\newglossaryentry{quotinetmap}{name={\ensuremath{\mathtt e}},description={quotient map for the equivalence relation $E$}}
\newglossaryentry{genfunc}{name={\ensuremath{f}},description={generic function on $X$}}
\newglossaryentry{gamma}{name={\ensuremath{\gamma}},description={Lipschitz curves on $X$}}
\newglossaryentry{dotgamma}{name={\ensuremath{\dot{|\gamma|}}},description={metric derivative of $\gamma$}}
\newglossaryentry{Domf}{name={\ensuremath{\Dom(f)}},description={domain of the function $f$}}
\newglossaryentry{equivE}{name={\ensuremath{E}},description={equivalence relation}}
\newglossaryentry{Fsigma}{name={\ensuremath{F_\sigma}},description={countable union of closed sets}}
\newglossaryentry{Gdelta}{name={\ensuremath{G_\delta}},description={countable intersection of open sets}}
\newglossaryentry{Gamma}{name={\ensuremath{\Gamma}},description={space of Lipschitz curves on $X$}}
\newglossaryentry{Gammac}{name={\ensuremath{\Gamma_c}},description={space of closed Lipschitz curves on $X$}}
\newglossaryentry{GammaA}{name={\ensuremath{\Gamma_A}},description={set of all subcurves of a given family of curves}}
\newglossaryentry{GammaApiS}{name={\ensuremath{\Gamma_{\bar A_{\bar \pi},\bar S}}},description={subcurves of $\bar A_{\bar \pi}$ passing trhough $\bar S$}}
\newglossaryentry{geoxy}{name={\ensuremath{\protect\overrightarrow{[x,y]}}},description={geodesic connecting $x$ to $y$}}
\newglossaryentry{Graphf}{name={\ensuremath{\Graph(f)}},description={Graph of a function $f$}}
\newglossaryentry{Hasdmeas}{name={\ensuremath{\mathcal H^d}},description={$d$-dimensional Hausdorff measure}}
\newglossaryentry{InteCirc}{name={\ensuremath{\mathcal I_c}},description={set of closed path with integer circuitation}}
\newglossaryentry{K13}{name={\ensuremath{K_{1/3}}},description={$1/3$ Cantor set in $[0,1]$}}
\newglossaryentry{K(X)}{name={\ensuremath{\mathcal K(X)}},description={compact subsets of $X$, with the Haudorff topology}}
\newglossaryentry{Lp}{name={\ensuremath{L^p(\mu)}},description={standard integral space of functions}}
\newglossaryentry{Lengthgam}{name={\ensuremath{L(\gamma)}},description={length of the Lipschitz curve $\gamma$}}
\newglossaryentry{LipX}{name={\ensuremath{\mathrm{Lip}([0,1];X)}},description={Lipschitz functions from $[0,1]$ to $X$}}
\newglossaryentry{Lambda}{name={\ensuremath{\Lambda_\gamma(t,v)}},description={evaluation of the cotangent vector field $v$ on the curve $\gamma$}}
\newglossaryentry{L1T*X}{name={\ensuremath{L^1(T^*X)}},description={cotangent bundle of $L^1$ covectors}}
\newglossaryentry{Lebesmeasu}{name={\ensuremath{\mathcal L^d}},description={Lebesgue measure in $\R^d$}}
\newglossaryentry{m}{name={\ensuremath{m}},description={probability on $X$ used for defining $W^{1,2}$}}
\newglossaryentry{calM}{name={\ensuremath{\mathcal M}},description={(complete) $\sigma$-algebra, often identificated with the corresponding measure algebra}}
\newglossaryentry{calME}{name={\ensuremath{\mathcal M^{E}}},description={$\sigma$-algebra made of $E$-saturated sets, of the corresponding measure algebra}}
\newglossaryentry{calMpi}{name={\ensuremath{\mathcal M^\pi}},description={maximal measure algebra generated by $\pi$}}
\newglossaryentry{mumeasure}{name={\ensuremath{\mu}},description={an element of $\mathcal P(X)$}}
\newglossaryentry{muy}{name={\ensuremath{\mu_y}},description={conditional probabilities of the disintegration of $\mu$}}
\newglossaryentry{Natural}{name={\ensuremath{\N}},description={natural numbers}}
\newglossaryentry{omega}{name={\ensuremath{\omega}},description={$L^\infty(m)$-function defined on $\bar S$ with values in $\mathbb S^1$}}
\newglossaryentry{pi}{name={\ensuremath{\pi}},description={test plan}}
\newglossaryentry{barpi}{name={\ensuremath{\bar \pi}},description={test plan generating a maximal equivalence relation}}
\newglossaryentry{hatpi}{name={\ensuremath{\hat \pi}},description={test plan on constant curves with $e(t)_\sharp \hat \pi = 1$}}
\newglossaryentry{probX}{name={\ensuremath{\mathcal P(X)}},description={probability measure over $X$}}
\newglossaryentry{projX}{name={\ensuremath{\mathtt p_X}},description={projection on the space $X$}}
\newglossaryentry{proji}{name={\ensuremath{\mathtt p_i}},description={projection on the $i$-th component of the space $\prod_i X_i$}}
\newglossaryentry{pushf}{name={\ensuremath{f_\sharp \mu}},description={push forward of $\mu$ by $f$}}
\newglossaryentry{Rangef}{name={\ensuremath{f(X)}},description={range of the function $f$}}
\newglossaryentry{Rational}{name={\ensuremath{\Q}},description={rational numbers}}
\newglossaryentry{Real}{name={\ensuremath{\R}},description={real numbers}}
\newglossaryentry{Revmap}{name={\ensuremath{R(\gamma)(t) = \gamma(1-t)}},description={time reversal map}}
\newglossaryentry{Resst}{name={\ensuremath{R_{s,t}(\gamma)}},description={restriction of the curve $\gamma$ to the interval $[s,t]$}}
\newglossaryentry{analytic}{name={\ensuremath{\Sigma^1_1}},description={analytic sets, i.e. projection of closed sets}}
\newglossaryentry{sigmanal}{name={\ensuremath{\sigma(\Sigma^1_1)}},description={$\sigma$-algebra generated by analytic sets}}
\newglossaryentry{SectionE}{name={\ensuremath{S}},description={section of the equivalence relation $E$}}
\newglossaryentry{S1}{name={\ensuremath{\mathbb{S}^1}},description={unitary circle in $\R^2$}}
\newglossaryentry{S1presob}{name={\ensuremath{S^1(X)}},description={space of functions with upper gradient $|Df|$ in $L^1(X)$}}
\newglossaryentry{s(pi)}{name={\ensuremath{s(\pi) = \{K_n\}_n}},description={selection of compact sets in $\mathcal A$}}
\newglossaryentry{barS}{name={\ensuremath{\bar S}},description={Borel section of an equivalence relation}}
\newglossaryentry{Tmapst}{name={\ensuremath{T_{s,t}(\gamma)}},description={closed curve oscillating between $\gamma(s)$ and $\gamma(t)$}}
\newglossaryentry{LinfTX}{name={\ensuremath{L^\infty(TX)}},description={tangent bundle of $L^\infty$ vector fields}}
\newglossaryentry{vcvector}{name={\ensuremath{v}},description={element of the cotangent bundle $L^1(T^*X)$}}
\newglossaryentry{vitali}{name={\ensuremath{v:[0,1] \to [0,1]}},description={Vitali function or Devil staircase}}
\newglossaryentry{w}{name={\ensuremath{w}},description={solution of the equation $\dot w(\gamma) = i w(\gamma) \Lambda(\gamma)$}}
\newglossaryentry{W12}{name={\ensuremath{W^{1,2}(X)}},description={space of functions in $L^2(X)$ with upper gradient in $L^2(X)$}}
\newglossaryentry{Xmetric}{name={\ensuremath{X}},description={separable metric space}}
\newglossaryentry{xmetric}{name={\ensuremath{x}},description={generic point in the metrix space $X$}}
\newglossaryentry{Integer}{name={\ensuremath{\Z}},description={integer numbers}}
\newglossaryentry{varpi}{name={\ensuremath{\varpi}},description={image measure of the disintegration of $\mu$}}

\numberwithin{equation}{section} 

\theoremstyle{plain} \newtheorem*{theorem*}{Theorem}
\theoremstyle{plain} 
\theoremstyle{plain} \newtheorem*{mthm*}{Main Theorem}
\theoremstyle{plain} \newtheorem*{conjecture*}{Conjecture}
\theoremstyle{plain} 
\theoremstyle{plain} \newtheorem*{problem*}{Problem}

\title{Exact integrability conditions for cotangent vector fields}


\author{Stefano Bianchini}
\address{S. Bianchini: S.I.S.S.A., via Bonomea 265, 34136 Trieste, Italy}
\email{bianchin@sissa.it}

\begin{document}


\begin{abstract}
In Quantum Hydro-Dynamics the following problem is relevant: let $(\sqrt{\rho},\Lambda) \in H^1(\R^d) \times L^2(\R^d,\R^d)$ be a finite energy hydrodynamics state, i.e. $\Lambda = 0$ when $\rho = 0$ and
\begin{equation*}
E = \int_{\R^d} \frac{1}{2} \big| \nabla \sqrt{\rho} \big|^2 + \frac{1}{2} \Lambda^2 \mathcal L^d < \infty.
\end{equation*}
The question is under which conditions there exists a wave function $\psi \in H^1(\R^d,\C)$ such that
\begin{equation*}
\sqrt{\rho} = |\psi|, \quad J = \sqrt{\rho} \Lambda = \Im \big( \bar \psi \nabla \psi).
\end{equation*}
The second equation gives for $\psi = \sqrt{\rho} w$ smooth, $|w| = 1$, that $i \Lambda = \sqrt{\rho} \bar w \nabla w$.

Interpreting $\rho \mathcal L^d$ as a measure in the metric space $\R^d$, this question can be stated in generality as follows: given metric measure space $(X,d,\mu)$ and a cotangent vector field $v \in L^2(T^* X)$, is there a function $w \in H^1(\mu,\mathbb S^1)$ such that
\begin{equation*}
dw = i w v.
\end{equation*} 

We show that under some assumptions on the metric measure space $(X,d,\mu)$ (conditions which are verified on Riemann manifolds with the measure $\mu = \rho \mathrm{Vol}$ or more generally on non-branching $MCP(K,N)$), we show that the necessary and sufficient conditions for the existence of $w$ is that (in the case of differentiable manifold)
\begin{equation*}
\int v(\gamma(t)) \cdot \dot \gamma (t) dt \in 2\pi \Z
\end{equation*}
for $\pi$-a.e. $\gamma$, where $\pi$ is a test plan supported on closed curves. This condition generalizes the conditions that the vorticity is quantized. We also give a representation of every possible solution.

In particular, we deduce that the wave function $\psi = \sqrt{\rho} w$ is in $W^{1,2}(X)$ whenever $\sqrt{\rho} \in W^{1,2}(X)$.
\end{abstract}


\thanks{I wish to thank Paolo Antonelli and Piero Marcati for introducing me to the beautiful subject of QHD}

\maketitle


\begin{center}
Preprint SISSA 18/2021/MATE
\end{center}

\tableofcontents

\section{Introduction}
\label{S:intro_wavef}

This paper is devoted to the following question arising in Quantum Hydro-Dynamics (QHD): assume that $(\sqrt{\rho},\Lambda) \in H^1(\R^d) \times L^2(\R^d,\R^d)$ are such that $\Lambda = 0$ when $\rho = 0$. 
The question is under which conditions there exists a wave function $\psi \in H^1(\R^d,\C)$ such that
\begin{equation}
\label{Equa:rho_Lambda_psi}
\sqrt{\rho} = |\psi|, \quad J = \sqrt{\rho} \Lambda = \Im \big( \bar \psi \nabla \psi).
\end{equation}
The second equation gives for $\psi = \sqrt{\rho} w$ smooth, $|w| = 1$, that $i \Lambda = \sqrt{\rho} \bar w \nabla w$. The interest for the above question arises because $\rho, J$ are obtained as solutions to the QHD system
\begin{equation}
\label{Equa:QHD}
\begin{cases}
\partial_t \rho + \dive J = 0, \\
{\displaystyle \partial_t J + \dive \big( \Lambda \otimes \Lambda + p(\rho) \Id \big) = \dive \bigg( \frac{\nabla^2 \rho}{4} - \nabla \sqrt{\rho} \otimes \nabla \sqrt{\rho} \bigg),}
\end{cases}
\end{equation}
which is the system of PDE satisfied by the hydrodynamics quantities in \eqref{Equa:rho_Lambda_psi} if $\psi \in C(\R,H^1(\R^d,\C))$ solves the Schr\"odinger equation
\begin{equation}
\label{Equa:Schrodinger}
i \partial_t \psi =  - \frac{1}{2} \Delta \psi + f'(|\psi|^2) \psi,
\end{equation}
where
\begin{equation*}
f(\rho) = \rho \int_{\rho*}^\rho \frac{p(s) - p(\rho*)}{s^2} ds, \quad \rho^* \text{ reference density}.
\end{equation*}
Apart from the 1d case, where a weak solution to \eqref{Equa:QHD} can be constructed and then the wave function $\psi$ is recovered (see \cite{AntMarZhe:1dQHD}), in the general case one first obtains the solution to \eqref{Equa:Schrodinger} and then shows that the associated hydrodynamics quantities \eqref{Equa:rho_Lambda_psi} satisfy \eqref{Equa:QHD}; for an overview of the mathematical theory for QHD we refer to \cite{AntMar:noteQHD}.

An implication that $(\rho,J)$ are derived from a wave function $\psi$ is that
\begin{equation}
\label{Equa:curl_1}
\mathrm{curl}\, J = 2 \nabla \sqrt{\rho} \wedge \Lambda, 
\end{equation}
which mimics the fact that $\mathrm{curl}\, \nabla \theta = 0$. Being this condition a local condition, it is not sufficient for the existence of a wave function $\psi$, as one can see with the elementary example in $\R^2$
\begin{equation}
\label{Equa:counter_J_curl}
\rho(x) \in C^1_c(\R^2), \rho(0) = 0, \qquad J(x) = \rho(x) \frac{\alpha x^\perp}{|x|^2}, \ \alpha \notin \Z.
\end{equation}
Indeed \eqref{Equa:curl_1} is verified and being $v = \alpha x^\perp/|x|^2$, $\mathrm{curl}\,v = 2 \pi \alpha \delta_0$, for $\alpha \notin \N$ it is impossible to solve
\begin{equation*}
D w = i w v,
\end{equation*}
where $Dw$ is the differential of the function $w$.

In \cite{AntMarZhe:intriQHD} the authors present some results which extend the 1d case recontruction of the wave function $\psi$ to the case when $x \in \R^d$ and $\rho,J$ are radially symmetric, and more interestingly when $x \in \R^2$, the mass $\rho$ is continuous and the vacuum set $\{\rho = 0\}$ is given by isolated points $\{x_1,x_2,\dots\}$, and the velocity field $v = J/\rho$ is a distribution such that
\begin{equation}
\label{Equa:integer_cirl}
\mathrm{curl} \, v = 2\pi \sum_j k_j \delta_{x_j}, \qquad k_j \in \Z.
\end{equation}
In this case the local function $w(x)$ given by $D w = iv w$ (which is obtained by the above curl-free condition) can be extended to a global one because of the \emph{integer circuitation condition} \eqref{Equa:integer_cirl}: indeed, considering a single singular point $x_1 = 0$
\begin{equation*}
\mathrm{curl}\, v = 2\pi k \delta_0,
\end{equation*}
one deduce that for all test function $\phi$
\begin{equation*}
- \int \nabla^\perp \phi \cdot v \mathcal L^2 = 2\pi k \phi(0).
\end{equation*}
In particular (by considering radial function $\psi$, for example) one deduce that for $\mathcal L^1$-a.e. $r$ it holds
\begin{equation}
\label{Equa:weak_curl}
\int_0^1 v(\gamma_r(t)) \cdot \dot \gamma_r(t) dt = 2 \pi k, \quad \gamma_r(t) = r (\cos(2\pi t),\sin(2\pi t)).
\end{equation} 
Hence, the local functions $w$ can be extended globally, since after a rotation around $x_1$ the value 
\begin{equation}
\label{Equa:inter_simple}
w(\gamma_r(1)) = w(\gamma_r(0)) e^{i \int_0^1  v(\gamma_r(t)) \cdot \dot \gamma_r(t) dt} = w(\gamma_r(0))
\end{equation}
is the same, and indeed the counterexample \eqref{Equa:counter_J_curl} does not satisfy exactly \eqref{Equa:integer_cirl} about the point $x=0$.

With more generality, one can see that the condition \eqref{Equa:integer_cirl} is equivalent to the following. Let $\Gamma(\R^2)$ be the metric space of Lipschitz curves $\gamma : [0,1] \mapsto \R^2$ with the uniform topology $C_0$, and define a \emph{test plan $\pi$} as a probability measure such that its evaluation $e(t)_\sharp \pi$ at time $t$ satisfies
\begin{equation*}
\int \phi(x) (e(t)_\sharp \pi)(dx) = \int \phi(\gamma(t)) \pi(d\gamma) \leq C \int \phi(x) dx,
\end{equation*}
for all $\phi \geq 0$ continuous function. The formula \eqref{Equa:inter_simple} is then rewritten as follows: for $\pi$-a.e. $\gamma$ such that $\gamma(0) = \gamma(1)$ (i.e. it is a closed path) it holds
\begin{equation}
\label{Equa:circute}
\int_0^1 v(\gamma(t)) \cdot \dot \gamma(t) dt = 2\pi \sum_j k_j \mathrm{Wi}(\gamma,x_j) \in 2\pi\Z,
\end{equation}
where $\mathrm{Wi}(\gamma,x)$ is the \emph{winding number} of $\gamma$ around $z \notin \gamma([0,1])$. Being the set $\{\rho=0\}$ locally finite (or more generally of $\mathcal H^1$-measure $0$), one can prove that $\pi$-a.e. $\gamma$ is not passing through $\{\rho = 0\}$, so that for $\pi$-a.e. $\gamma$ the winding number $\mathrm{Wi}(\gamma,x_i)$ is defined for all $x_i$. The above formula is actually equivalent to the weak formulation \eqref{Equa:weak_curl} because for sufficiently regular functions the boundary of almost all level sets is the union of closed curves.

Formula \eqref{Equa:circute} is meaningful in any dimension, and actually even in metric measure spaces, and gives rise to the following questions.

\begin{problem*}
\label{Prob:circ_1}
Let $(X,d,\mu)$ be a metric measure space (a Polish space with $\mu \in \mathcal P(X)$ Borel probability measure) and let $v \in L^2(T^* X)$ be a cotangent vector field such that, if $\pi$ is a test plan supported on the set of closed curves, then
\begin{equation}
\label{Equa:integer_circ_32}
\int_0^1 v(\gamma(t)) \cdot \dot \gamma(t) dt \in 2\pi\Z \qquad \text{for $\pi$-a.e. $\gamma$.}
\end{equation}
\begin{enumerate}
\item Is there a function $w \in L^\infty(X,\mathbb S^1)$ such that its differential satisfies
\begin{equation}
\label{Equa:Diffe:relw}
Dw = i w v \, ?
\end{equation}
\item Assuming moreover that $m$ is the ambient measure in $(X,d)$ and
\begin{equation}
\label{Equa:meaure_mu}
\mu = \frac{|f|^2 m}{\|f\|_2^2} \quad \text{with} \ f \in W^{1,2}(X,\R),
\end{equation}
does the function $f w$ belong to $W^{1,2}(X,\C)$?
\end{enumerate}
\end{problem*}

Above and in the following we will use the identification $\{z \in \C,|z|=1\} = \mathbb S^1$, this should not generate confusion. In the case $X = \R^d$ (or a differentiable manifold) the objects considered in the previous statement are standard, while in Section \ref{S:setting} we recall their definitions for generic metric measure spaces following \cite{giglipasqualetto:NSG}.

The assumption that $v \in L^2(T^*X)$ is the analog of the requirement $\Lambda = \sqrt{\rho} v \in L^2(\R^d)$ in the Euclidean setting. We will call the condition \eqref{Equa:integer_circ_32} the \emph{integer circuitation condition}. It is obvious that the opposite is trivial, i.e. $v= -i Dw/w$ has integer circuitation.

It turns out that in general metric spaces the answer is no, simple because the space may have too few curves and the measure $\mu$ may be too singular (Example \ref{Ex:notgoodsol}). But in more regular spaces the answer is positive:

\begin{theorem}
\label{Theo:gen_case}
Assume that $(X,d,m)$ has the measure contraction property $MCP(K,N)$ (and if branching the ambient measure $m$ is the Hausdorff measure $\mathcal H^N$). Then, for every $f \in W^{1,2}(X,\R)$ and cotangent vector field $v \in L^2(T^*X)$ w.r.t. the measure $\mu = |f|^2 m/\|f\|_2^2$, there exists $w \in L^\infty(X,\mathbb S^1)$ such that 
\begin{equation}
\label{Equa:cicuit_dP}
Dw = iw v \quad \text{and} \quad fw \in W^{1,2}(X,\C).
\end{equation}
Moreover, there exists a partition of $\{|f| > 0\}$ into at most countably many sets $\Omega_i$, $m(\Omega_i) > 0$,
\begin{equation*}
\{|f| > 0\} = \bigcup_{i \in \N} \Omega_i,
\end{equation*}
such that for every function $w'$ satisfying $Dw = i w v$ it holds
\begin{equation*}
\frac{w'(x)}{w(x)} = c_i \in \mathbb S^1 \quad \text{constant for $m$-a.e. $x \in \Omega_i$}. 
\end{equation*}
\end{theorem}

The second part of the statement is the same as the 1d case, where in every interval such that $\{|f| > 0\}$ the function $w$ is determined up to a constant. Moreover, every Riemann manifold satisfies the assumptions of the theorem. 

\smallskip

Before giving an overview of the proof, some comments are in order.

\smallskip

The actual result proved here is more general: the idea is to decompose the space $X$ into regions $\{\Omega_\alpha\}_\alpha$ in which a.e. points $x,y \in \Omega_\alpha$ can be connected by test plans, and in these regions to construct the function $w$ by integrating along curves. The technical part is to remove curves $\gamma$ which are not "good" for $v$: either the integral along $\gamma$ is not defined or the circuitation \eqref{Equa:cicuit_dP} is not integer if $\gamma$ is closed. The non technical part is that such a decomposition does not need to be sufficiently regular to decompose the metric measure space $(X,d,m)$ into metric measure spaces $(\Omega_\alpha,d,m_\alpha)$. Example \ref{Ex:notgoodsol} indeed uses the well known construction of the Vitali set to show that the decomposition above does not exist in general. From a measure theoretical point of view, one needs to assume that the partition $\{\Omega_\alpha\}_\alpha$ admits a \emph{strongly consistent disintegration}, i.e. one can write
\begin{equation}
\label{Equa:conncc}
\mu = \int \mu_\alpha \omega(d\alpha) \quad \text{with} \quad \mu_\alpha \in \mathcal P(\Omega_\alpha).
\end{equation}
This condition is certainly fullfilled if the partition $\{\Omega_\alpha\}_\alpha$ is countable, and this countability property is true under the assumptions of Theorem \ref{Theo:gen_case}. It is not clear to us if such a decomposition into these measure-theoretic indecomposable components $\Omega_\alpha$ has some interests on its own.

\smallskip

Within these setting, the function $v$ is not locally integrable, so $\mathrm{curl}\,v$ is meaningless even distributionally. On the other hand, as in the example of \eqref{Equa:counter_J_curl}, the study of $\mathrm{curl}\,\Lambda$ does not yield any condition into the existence of the function $w$. It is not clear if there exists some $\mathrm{curl}$-like functional which allows to write the integer circuitation condition in a more PDE point of view, for example identifying the quantum vortex structure on which the rotational part of $v$ is concentrated and has $2\pi\Z$-density, as in \eqref{Equa:integer_cirl}.

\smallskip

Another question is whether the integer circuitation assumption \eqref{Equa:integer_circ_32} is preserved by the flow of \eqref{Equa:QHD}. A tautological argument is that since a solution to the QHD system is constructed by means of a wave function $\psi$ solving Shcr\"odinger equation, a posteriori the integer circuitation condition must be compatible with the QHD formulation. Similarly for the study of the identity \eqref{Equa:curl_1}, it can be shown that the integer circuitation condition is compatible for smooth solutions of \eqref{Equa:QHD}. Indeed one rewrites the system \eqref{Equa:QHD} as
\begin{equation}
\label{Equa:smoot_QHD}
\begin{cases}
\partial_t \rho + v \cdot \nabla \rho = 0, \\
{\displaystyle \partial_t v + v \cdot \nabla v = \nabla \bigg( \frac{\Delta \sqrt{\rho}}{2\sqrt{\rho}} - f'(\rho) \bigg),}
\end{cases}
\end{equation}
so that both $\rho$ and $v$ are transported by the same vector field $v$. Hence any curve $\gamma$ at time $t$ is transported back to a curve $\gamma'$ at time $0$ (being the flow $\dot X = v(t,X)$ locally smooth), and defining
\begin{equation*}
\gamma_t(\alpha) = X(t,\gamma(\alpha))
\end{equation*}
one has
\begin{equation*}
\begin{split}
\frac{d}{dt} \int v(t,\gamma_t(\alpha)) \cdot \dot \gamma_t(\alpha) d\alpha \bigg|_{t=0} &= \int \nabla \bigg( - \frac{|v_0|^2}{2} + \frac{\Delta \sqrt{\rho_0}}{2\sqrt{\rho_0}} - f'(\rho_0) \bigg) \cdot \dot \gamma(\alpha) d\alpha \\
& \quad + \int v_0(\gamma(\alpha)) \cdot \nabla v_0(\gamma(\alpha)) \dot \gamma(\alpha) d\alpha \\
&= \int \nabla \bigg( \frac{\Delta \sqrt{\rho_0}}{2\sqrt{\rho_0}} - f'(\rho_0) \bigg) \cdot \dot \gamma(\alpha) d\alpha = 0.
\end{split}
\end{equation*}
Hence the fact that the l.h.s. of the equation for $v$ in \eqref{Equa:smoot_QHD} is a gradient yields that the integer circuitation condition is preserved.

\smallskip

The last two remarks (i.e. if one can give a meaning to $\mathrm{curl}\,v$ and if the computation above can be performed for nonsmooth solutions) are related to the fundamental question of regularity of the QHD solution, which is echoed by the regularity of the flow of the transport equation with vector field $v$. As far as we know, these questions are completely open in the general case and they are deeply connected to the proof of existence of a solution by purely fluid dynamics techniques.

\smallskip

As a last observation, a completely similar analysis can be done for the question:
\begin{equation*}
\text{given $v \in L^1(T^*X)$, there exists a function $w$ such that $Dw = v$, i.e. the form is exact?}
\end{equation*}
The results and the proof are exactly the same, by replacing the integer circuitation condition with the \emph{circuitation free condition}: if $\pi$ is a test plan supported on the set of closed curves, then
\begin{equation}
\label{Equa:free_circ_32}
\int_0^1 v(\gamma(t)) \cdot \dot \gamma(t) dt = 0 \qquad \text{for $\pi$-a.e. $\gamma$.}
\end{equation}

\subsection{Sketch of the proof and structure of the paper}
\label{Ss:sketch_proof_intro}

In a nutshell, the idea of constructing $w$ is the same as for the smooth case: given a value $w(\bar x)$ as some point $\bar x$, find $w(x)$ by solving the ODE
\begin{equation}
\label{Equa:ODE:curev}
\frac{d}{dt} w(\gamma(t)) = i w(\gamma(t)) v(\gamma(t)) \cdot \dot \gamma(t), \qquad w(\gamma(0)) = w(\bar x), \ \gamma(0) = \bar x, \gamma(1) = x.
\end{equation}
The integer circuitation condition should take care of the possibility of having multiple curves connecting $\bar x$ to $x$.

The problem is that all definitions are given in terms of measures, and then one has to consider only the good path, i.e. the ones for which the circuitation condition holds. This difficulty occurs also for $X = \R^d$, because of the density $f \in W^{1,2}$. Given a closed curve $\gamma$ such that $\gamma(t) = \gamma(s)$ for some $0 < s < t < 1$, it is not even true in general that the restricted closed curve $R_{s,t} \gamma = \gamma \llcorner_{[s,t]}$ is of integer circuitation, even if $\gamma$ is. The first results (Proposition \ref{Prop:self_inters} and Theorem \ref{theo:neglinters}) state that under the integer circuitation condition \eqref{Equa:integer_circ_32} every test plan is concentrated on a set of curves $A_\pi$ such that every closed curve obtained by patching pieces of finitely many curves of $A_\pi$ has integer circuitation. In particular, restrictions of curves satisfy \eqref{Equa:integer_circ_32}. We call the curves obtained by piecing together parts of other curves \emph{subcurve} or \emph{subloops} if closed.

The argument in $\R^d$ (and in differentiable manifolds) is much easier than the general case in MMS: the case $X = \R^2$ is particularly straightforward (Section \ref{S:case_R2}). Indeed, one can use test plans to move mass along the $i$-axis, $i=1,2$, and then construct a finite dimensional map from $\R^{4}$ into the set of curves $\Gamma$ such that every point in some Cantor-like neighborhood is connected with another one by a specific curve moving first along the $1$-direction and then along the $2$-direction (Lemma \ref{Lem:test_R2}). The integer circuitation condition and the result of the previous part imply that we can remove a set of measure $0$ in $\R^{4}$ and the remaining curves have integer circuitation together with the subcurves (Corollary \ref{Cor:every_traj_good}). Since by $W^{1,2}$-regularity of the function $f$ one can  locally assume that $f > 1/2$ on sufficiently many of these segments, the function $w$ can be constructed on this set (which is of positive $\mathcal L^2$-measure, Lemma \ref{Lem:w_funct_good}). This procedure can be repeated in every Lebesgue point of $f$ and $Df$, and the different $w$ (which are defined up to a constant) can be related with each other by a test plan (if it exists) concentrated on curves with initial points in one set and end points in the other set, making them compatible (Lemma \ref{Lem:single_constat}). The same analysis can be done in $\R^d$, $d \geq 3$, at the expense of a more complicated selection procedure for the curves. Since we will address the problem in general spaces, we will leave this extension to the reader which wants to avoid some measure theoretic arguments used in the remaining parts.

The observation above is that locally one can construct suitable paths which have integer circuitation, where $f > 0$. In particular these sets cannot be further divided by just removing a set of measure $0$ trajectories of the above plan. The idea is then to find this minimal partition (see Section \ref{S:constr_min_eq} for the abstract theory, and Section \ref{Ss:appli_test} for the application to out case). \\
Given a set of curves $\{\gamma_\alpha\}_\alpha$, define the minimal (by set inclusion) equivalence relation $E$ by requiring that $\gamma_\alpha(t) E \gamma_\alpha(s)$ for all $s,t \in [0,1]$, $\gamma_\alpha$ (Definition \ref{Def:equiv_rel_ch}). If we have a test plan $\pi$, for every set of curves $A_\pi$, $\pi(A_\pi) = 1$, one has clearly different equivalence relations $E^{A_\pi}$. However the space $X$ is separable, and we use this fact to deduce that there is a set $\bar A_\pi$ such that the equivalence relation $E^{\bar A_\pi}$ is minimal: this means that if $\mathcal M^{\bar A_\pi}$ is the family of $\mu$-measurable sets which are saturated w.r.t. $E^{\bar A_\pi}$, then for all other sets $A_\pi$ where $\pi$ is concentrated there exists a $\mu$-conegligible set where the $\sigma$-algebra generated by $E^{A_\pi}$ is a subalgebra of $\mathcal M^{E^{\bar A_\pi}}$ (Proposition \ref{Prop:exi_min_pi}). This does not means that every atom of $\mathcal M^{\bar A_\pi}$ coincides with an equivalence class of $E^{\bar A_\pi}$: the situation is completely equivalent to the Vitali set of $[0,1]$, where $E = [0,1]/\Q$ and then the $E$-saturated $\mathcal L^1$-measurable sets have Lebesgue measure $0$ or $1$ (see Example \ref{Ex:notgoodsol}).

Notwithstanding the counterexample right above, the next step we perform is 
to find a test plan $\bar \pi$ such that the maximal measure algebra $\mathcal M^{\bar \pi}$ is the minimal among all measure algebras $\mathcal M^\pi$ (Proposition \ref{Prop:perf_test}). The reason for this analysis is the following: first, being $\mathcal M^{\bar \pi}$ maximal, up to a negligible set the space $X$ cannot be further divided into smaller atoms, i.e. the atoms of $\mathcal M^{\bar \pi}$ are the minimal components used by the test plan $\bar \pi$. Being $\mathcal M^{\bar \pi}$ minimal among the maximal measure algebras, we deduce that the all others test plans use curves which are contained into atoms of $\mathcal M^{\bar \pi}$ (up to a negligible set as usual, see Corollary \ref{Cor:set_good}). These components are in some sense the minimal connected-by-curves components of $(X,\mu)$, in the sense that one cannot find a further decomposition such that all test plans have trajectories concentrated on the atoms of this new refined decomposition.

Unfortunately finding the minimal-maximal $\sigma$-algebra is not sufficient: indeed the problem is that in general the components of the minimal-maximal decomposition do not coincide with the equivalence classes of the generating relation (as observed above, and as Example \ref{Ex:notgoodsol} shows). This condition seems essential and it is assumed in the construction of $w$ (Assumption \ref{Ass:disi_cons}). In the case of sufficiently regular MMS (in particular $MCP(K,N)$-spaces where the ambient measure is $f\mathcal H^N$ with $f \in W^{1,2}(X,\mathcal H^N)$), it is satisfied (Example \ref{Ex:torus_Riemann}). Under the assumption that the equivalence classes coincide with the atoms of the measure algebra $\mathcal M^{\bar \pi}$, the construction is now fairly simple, and follows the line for the $\R^2$ case: for the test plan $\bar \pi$ take the set $A^{\bar \pi}$ such that every subloop is of integer circuitation, choose a Borel section and a Borel function $\omega$ on the quotient space and for every $x$ compute $w$ by solving the ODE \eqref{Equa:ODE:curev} with initial data $\omega(y)$ along a path connecting $y$ to $x$, where $y$ is the point in Borel section in the equivalence class of $x$ (Proposition \ref{Prop:exist_w}). It is important to note that the existence of such a path follows because we assume that the atoms of the measure algebra coincide with the equivalence classes, or equivalently that the disintegration is strongly consistent \eqref{Equa:conncc}. The fact that $\bar A^{\bar \pi}$ is of free circuitation gives that the value $w(x)$ is independent of the particular path chosen. The same argument yields that any other solution $w'$ is such that $w'/w$ is constant on the equivalence class (Theorem \ref{Theo:w_final}).

The final analysis of the case when $\mu = f m$, with $f \in W^{1,2}(X,m)$ is at this point a corollary, once it is assumed that the space $(X,d,m)$ has a test plan connecting all points at lest locally (a \emph{Democratic test plan}, as defined in \cite{lottvillani:weakcur}, see Assumption \ref{Assu:local_test}): indeed in this case the equivalence classes have positive measure, and then the disintegration is strongly consistent. Metric measure spaces enough regular satisfy this assumption (Example \ref{Ex:torus_Riemann}). For the existence of $w$ solving \eqref{Equa:Diffe:relw}, the only thing to verify is that there exists a test plan $\bar \pi$ satisfying the assumptions of the previous sections, and this is true in these spaces (Lemma \ref{Lem:one_class}). Finally, to check that $fw \in W^{1,2}(X,m)$ it is enough to prove that that it holds along test plans for $m$, and this is fairly easy once one have a Borel way of passing from test plans for $m$ to test plans for $\mu$ (Lemma \ref{Lem:Borel_H_1} and Proposition \ref{Prop}).

\section{Notation}
\label{S:notation}

The real numbers are denoted by \gls{Real}, the rational numbers by \gls{Rational}, the natural numbers by \gls{Natural} and the integers by \gls{Integer}. The unitary circle in $\R^2$ is denoted by \gls{S1}. The cardinality of a set $A$ is denoted by \gls{cardA}.

A generic constant (which may vary from line to line) is denoted with \gls{Cconst} or $C$, eventually with an index or accent.

Let \gls{Xmetric} be a separable complete metric space with distance \gls{ddistance}. A point in $X$ will be denoted by \gls{xmetric}. The balls of radius $r$ centered at $x$ is denoted as \gls{Ball}, and more generally the $r$ neighborhood of a set $A$ as \gls{BallSet}. The family of compact subsets of a metric space $X$ is denoted by \gls{K(X)}, and we will use the Hausdorff distance \gls{Hausd}. Countable intersection of open sets are \gls{Gdelta}-sets, while countable union of closed sets are \gls{Fsigma}-sets. The projection onto $X$ is written as \gls{projX}, or sometimes \gls{proji} if it is the $i$-component of $\prod_j X_j$. The Borel $\sigma$-algebra is \gls{Borelsig}, the projection of closed sets are the analytic sets \gls{analytic}, and the $\sigma$-algebra generated by them is \gls{sigmanal}.

In case there is no ambiguity, the geodesic connecting $x$ to $y$ is $\gls{geoxy}$: this notation will be used mainly in $\R^d$.

A generic function from $X$ to $Y$ will be denoted by $\gls{genfunc} : X \to Y$, and we will call $f$ a Borel function when it is a Borel function whose domain of definition \gls{Domf} is a Borel set. The graph of a function $f : X \to Y$ is denoted by
\begin{equation*}
\gls{Graphf} = \big\{ (x,f(x)), x \in X \big\} \subset X \times Y,
\end{equation*}
and its range by 
\begin{equation*}
\gls{Rangef} = \big\{ f(x), x \in X \big\}.
\end{equation*}
For a (multi-valued) function $f:X \to Y$ we will denote $f(x)$ both as the image of $f(x)$ or as the section $\Graph f \cap \{x\} \times Y$. In general this will not generate confusion.

The space of Lipschitz functions \gls{LipX} from $[0,1]$ into $X$ is denoted by \gls{Gamma}: we will refer to it as the set of Lipschitz curves on $X$. Its subset of closed Lipschitz curves (or closed loops or loops for shortness) will be written as \gls{Gammac}. An element of both sets will be written as \gls{gamma}. Explicitly
\begin{equation*}
\Gamma = \Lip([0,1];X), \quad \Gamma_c = \Big\{ \gamma \in \Lip([0,1];X), \gamma(0) = \gamma(1) \Big\}.
\end{equation*}
It is known that $\Gamma$ is a separable metric space with the $C^0$-topology \cite[Theorem 2.4.3]{Sri:courseborel}. 

The evaluation map $\gls{evaluation} = e(t,\gamma)$ is the map
\begin{equation*}
\Gamma \ni \gamma \mapsto e(t,\gamma) = \gamma(t) \in X.
\end{equation*}
The restriction operator \gls{Resst} is defined as the restriction of the curve $\gamma$ to the interval $[s,t] \subset [0,1]$,
\begin{equation*}
R_{s,t} (\gamma)(\tau) = \gamma\big( (1-\tau) s + \tau t \big), \quad 0 \leq s \leq t \leq 1.
\end{equation*}
The time inversion of a curve is denoted by \gls{Revmap}.

The metric derivative of $\gamma$ is denoted by \gls{dotgamma}. The length of a curve $\gamma \in \Gamma$ is defined as
\begin{equation*}
\gls{Lengthgam} = \int_0^1 \dot{|\gamma|}(t) dt.
\end{equation*}

Given a set $A \subset \Gamma$, we say that $\gamma \in \Gamma$ is a sub-curve of $A$ if there are finitely many curves $\gamma_i \in A$ such that
\begin{equation*}
\Graph \gamma \subset \bigcup_i \Graph \gamma_i.
\end{equation*}
The set of all sub-curves of $A$ is denoted by \gls{GammaA}.

The Lebesgue measure in $\R^d$ is denoted with \gls{Lebesmeasu} and the $d$-dimensional Hausdorff measure by \gls{Hasdmeas}. Let \gls{probX} be the space of probability measures over $X$ and \gls{m},\gls{mumeasure} elements of $\mathcal P(X)$: in the second part of our analysis $\mu = fm$ (and $m$ will be used as ambient probability for the space \gls{W12}), but in the first part we will consider $(X,d,\mu)$ as a given metric measure space. The topology on $\mathcal P(X)$ is the narrow topology. We denote the standard integral spaces by \gls{Lp}. The push forward of a measure $\mu$ by a map $f$ is written by \gls{pushf}. The disintegration of a probability according to a map $f : X \to Y$ will be denoted as
\begin{equation*}
\mu = \int \gls{muy} \varpi, \quad \gls{varpi} = f_\sharp \mu.
\end{equation*}
We will also write $\varpi = \mu \llcorner \mathcal M$, where $\mathcal M$ is the $\sigma$-algebra made of $\mu$-measurable sets saturated w.r.t. the equivalence relation $Ex = f^{-1}(f(x))$. 

If \gls{equivE} is an equivalence relation, we will write $xEx'$ if $x,x'$ are equivalent, and the equivalence class for $x$ as \gls{equivcla} or $E(x)$. We will not distinguish between $E$ and its graph in the product space. We will often use the letters $\alpha,\beta,\dots$ for the elements of the quotient space $X/E$. The quotient map is considered either as a section of $X$ (i.e. a set \gls{SectionE} containing a single point for every equivalence class) or as a map $\gls{quotinetmap} : X \to [0,1]$ whose level sets are the equivalence classes. 

\section{Settings and basic computations}
\label{S:setting}

In this section we introduce the setting for our problem: it is based on the analysis on metric measure space of \cite{gigli:nonsmootailored,giglipasqualetto:NSG}, with elementary slight variations.

\subsection{Test plans}
\label{Ss:test_plans}

\begin{definition}
\label{Def:admissi}
A measure $\gls{pi} \in \mathcal P(\Gamma)$ is a \emph{$\mu$-admissible test plan} (or \emph{test plan} when the reference measure $\mu$ is clear from the context) if for some constant $C$ it holds
\begin{equation}
\label{Equa:curves_pi}
\int_0^1 e(t)_\sharp (\dot{|\gamma|} \eta) dt \leq C \mu.
\end{equation}
The set of $\mu$-admissible test plans is denoted by \gls{adm}, and \gls{Ccompr} is called the \emph{compressibility constant} of $\pi$.
\end{definition}

The above formula \eqref{Equa:curves_pi} is equivalent to 
\begin{equation*}
\int \bigg[ \int_0^1 \ind_{\gamma(t) \in A}(t,\gamma) \dot{|\gamma|}(t) dt \bigg] \eta(d\gamma) \leq C \mu(A).
\end{equation*}
It is clear that the above definition is invariant if we reparameterize the curves $\gamma$. 
In the following we will often consider the parametrization of $\gamma$ by length, so that $\dot{|\gamma|} = 1$: this is to avoid to write the factor $\dot{|\gamma|}$ in many equations. In that case the interval of integration is $[0,L(\gamma)]$, and we will often not specify the interval of integration if clear from the context.

\begin{remark}
\label{Rem:variat}
For the analysis in this work we do not require the conditions
\begin{equation}
\label{Equa:variat}
\int \bigg[ \int_0^1 \big( \dot{|\gamma|}(t) \big)^2 dt \bigg] \pi(d\gamma) < \infty, \quad e(t)_\sharp \pi \leq C \mu,
\end{equation}
see for example \cite[Definition 2.1.1]{giglipasqualetto:NSG}. These assumptions 
will be analyzed in Proposition \ref{Prop:gene_test_plan}: here we can just say that since the integer circuitation condition \eqref{Equa:integer_circ_32} is required for $\pi$-a.e. $\gamma$ and is independent on the parametrization of $\gamma$, it is natural to consider test plans as in Definition \ref{Def:admissi}.
%
\end{remark}

\begin{remark}
\label{Rem:transfe}
An admissible test plan $\pi$ is given by the constant curves: define the map
\begin{equation*}
X \ni x \mapsto \gamma_x \in \Gamma_c, \qquad \gamma_x(t) = x \quad \forall t \in [0,1],
\end{equation*}
and set
\begin{equation}
\label{Equa:pi_f}
\pi = (\gamma_{x})_\sharp (f(x) \mu(dx))
\end{equation}
for every $f \in L^\infty(\mu)$, $\|f\|_1 = 1$. It is immediate from the definition to see that
\begin{equation*}
\pi(\Gamma) = 1, \quad e(t)_\sharp \pi = f\mu,
\end{equation*}
so that it is admissible. In particular the measure defined in \eqref{Equa:pi_f} when $f \equiv 1$ is denoted with \gls{hatpi}.

It may happens that these are the only test plans: for example, if $(X,d)$ is totally disconnected (or zero dimensional \cite[Section 2.2]{Sri:courseborel}), the only Lipschitz curves $\gamma : [0,1] \to X$ are the constant ones.
\end{remark}

It is immediate to see that if $\pi \in \adm(\mu)$ and $\gamma \mapsto (s(\gamma),t(\gamma))$ is a Borel function, then $(R_{s,t})_\sharp \pi, R_\sharp \pi \in \adm(\mu)$ because
\begin{equation*}
e(\tau)_\sharp \big( (R_{s,t})_\sharp \pi \big) = \gamma((1-\tau)s(\gamma) + \tau t(\gamma))_\sharp \pi, \quad e(\tau)_\sharp \big( R_\sharp \pi \big) = e(1-\tau)_\sharp \pi.
\end{equation*}

\subsection{Tangent and cotangent bundles}
\label{Ss:tangent_cotanget}

Here we recall the definition of tangent and cotangent bundles in a metric measure space $X$, following the approach of \cite{gigli:nonsmootailored,giglipasqualetto:NSG}.

Let \gls{S1presob} be the set of functions $f : X \to \R$ with the following property: there exists $g \in L^1(X)$ such that for all test plan $\pi$ 
\begin{equation*}
f \circ \gamma \ \text{is a.c. $\pi$-a.e. $\gamma$} \quad \wedge \quad e_\sharp \bigg( \bigg| \frac{d}{dt} f(\gamma(t)) \bigg| \mathcal L^1 \llcorner_{[0,1]} \times \pi \bigg) \leq C g \mu,  
\end{equation*}
where $C$ is the compressibility constant of $\pi$ (see \eqref{Equa:curves_pi}) and $e : [0,1] \times \Gamma \to X$ is the evaluation map. The essential infimum of all $g$'s is called the minimal upper gradient \gls{uppegr}.

Let \gls{vcvector} be an element of the cotangent bundle \gls{L1T*X}, which can be defined as the $L^1(\mu)$-normed $L^\infty(\mu)$-modulus closure of the set
\begin{equation*}
\bigcup_{n \in \N} \Big\{ \{(\Omega_i,f_i)\}_{i=1}^n, \{\Omega_i\}_{i=1}^n  \ \text{partition of} \ X, \{f_i\}_{i=1}^n \subset S^1(X) \Big\}/\sim
\end{equation*}
w.r.t. the norm
\begin{equation*}
\big\| \{(\Omega_i,f_i)\}_{i=1}^n \big\| = \sum_i \int_{\Omega_i} |Df_i| \mu,
\end{equation*}
where the equivalence relation $\sim$ is w.r.t. the above norm, and $|Df|$ is the minimal upper gradient (see \cite[Theorem 4.1.1]{giglipasqualetto:NSG}). We can associate to a function $f \in S^1(X)$ the cotangent vector field $\gls{df} = (X,f)$, so that
\begin{equation}
\label{Equa:tangnet_dense}
\{(\Omega_i,f_i)\}_{i=1}^n = \sum_{i=1}^n \ind_{\Omega_i} Df_i.
\end{equation}

The tangent bundle \gls{LinfTX} is defined as the dual $L^\infty(\mu)$-normed $L^\infty(\mu)$-modulus, i.e. the space of linear functional $L : L^1(T^*X) \to L^1(\mu)$ such that for $f \in L^\infty(\mu)$, $v \in L^1(T^*X)$
\begin{equation*}
L(f v) = f L(v), \quad \|L\| = \sup_{\|v\| \leq 1} \int |L(v)| \mu < \infty.
\end{equation*}
The above norm definition gives 
that the pointwise norm is
\begin{equation*}
|L| = \mathrm{ess\text{-}sup} \big\{ L(v), \|v\| \leq 1 \big\} \in L^\infty(\mu).
\end{equation*}

If $\pi$ is a test plan, then for every simple element $v = \sum_i \ind_{\Omega_i} Df_i$ we can define
\begin{equation}
\label{Equa:simple_plan_ct}
\Lambda_\gamma(t,v) = \sum_i \ind_{\Omega_i}(\gamma(t)) \frac{d}{dt} f_i(\gamma)(t) \quad (\dot{|\gamma|} = 1),
\end{equation}
which satisfies by \eqref{Equa:curves_pi}
\begin{equation*}
\int \bigg[ \int_0^{L(\gamma)} |\Lambda_\gamma(t,v)| dt \bigg] \eta(d\gamma) \leq \sum_i C \int_{\Omega_i} |Df_i| \mu = C \|v\|_{L^1(T^*X)}. 
\end{equation*}
Hence, being the elements of the form \eqref{Equa:tangnet_dense} dense in $L^1(T^*X)$ and $v \mapsto \Lambda_\gamma(t,v)$ linear, passing to the limit we obtain that for every $v \in L^1(T^*X)$ there exists a unique $L^1(\mathcal L^1 \times \pi)$-function $\Lambda_\gamma(t,v)$ such that
\begin{equation}
\label{Equa:plan_ct}
\gls{Lambda} = \lim_{v_n \to v} \Lambda_\gamma(t,v_n),
\end{equation}
where $v_n = \sum_i \ind_{\Omega_i} Df_i$ is converging to $v$ in $L^1(T^*X)$, and it satisfies
\begin{equation}
\label{Equa:norm_Lambda_gamma}
\|\Lambda_\cdot(\cdot,v)\|_{L^1(\mathcal L^1 \otimes \pi)} \leq C \|v\|_{L^1(T^*(X))}, \quad \text{$C$ compressibility constant}.
\end{equation}
Hence $L^1(T^*X) \ni v \mapsto \Lambda_\gamma(t,v) \in L^1(\mathcal L^1 \times \pi)$ is linear and continuous. 

\begin{remark}
\label{Rem:Lambda_pi}
The function $\Lambda_\gamma(\cdot,v)$ depends on $\pi$. Since we are using it only with a test plan which is clear from the context, we will not add the dependence w.r.t. $\pi$ in what follows.
\end{remark}


We will need the following property of $\Lambda_\gamma(t,v)$: we write the proof for completeness.

\begin{lemma}
\label{Lem:eta_gamma_Lambda}
There exists a $\mathcal L^1 \times \pi$-negligble set $N$ such that 
\begin{equation*}
\gamma(t) = \gamma(s) \quad \Rightarrow \quad \Lambda_\gamma(t,v) = \Lambda_\gamma(s,v)
\end{equation*}
for all $s,t \in [0,L(\gamma)] \setminus N(\gamma)$. 
%
\end{lemma}

In particular, the function $\Lambda_\gamma$ can be defined as an $L^1$-function of $\mathcal H^1 \llcorner_{\gamma(0,L(\gamma))}$. 

\begin{proof}
By \eqref{Equa:norm_Lambda_gamma} we have that if $v_n \to v$ in $L^1(T^*X)$ then $\Lambda_\gamma(t,v_n) \to \Lambda_\gamma(t,v)$ in $L^1(\mathcal L^1 \times \pi)$, so that up to subsequences it converges $\mathcal L^1 \times \pi$-a.e.. In particular, let $v_n \to v$ be a sequence of cotangent vector fields in $L^1(T^*X)$ such that the statement is true for each $v_n$, and let $N_n \subset [0,1] \times \Gamma$ be the corresponding $\mathcal L^1 \times \pi$-negligible set such that
\begin{equation*}
\forall s,t \notin N_n(\gamma) \ \Big( \gamma(t) = \gamma(s) \ \Rightarrow \ \Lambda_\gamma(s,v_n) = \Lambda_\gamma(t,v_n) \Big).
\end{equation*}
Clearly the set $N = \cup_n N_n$ is negligible. Since $\Lambda_\gamma(t,v_n) \to \Lambda_\gamma(t,v)$ up to a subsequence in a $\mathcal L^1 \times \pi$-conegligble set $A \subset \R \times \Gamma$, we obtain that for all $(t,\gamma),(s,\gamma) \in A \setminus N$ it holds
\begin{equation*}
\gamma(t) = \gamma(s) \ \Rightarrow \ \Lambda_\gamma(s,v) = \Lambda_\gamma(t,v).
\end{equation*}
It is thus sufficient to prove the statement is true for $v$ of the form \eqref{Equa:tangnet_dense}, and in particular for a single $f$.

Define the set of transversal self-intersections
\begin{equation*}
\begin{split}
N' &= \bigcup_n N'_{n} \\
&= \bigcup_{n \in \N} \bigg\{  (s,\gamma) : \exists t \in [0,L(\gamma)] \, \Big( \gamma(s + 2^{-n}(-1,1)) \cap \gamma(t + 2^{-n}(-1,1)) = \{\gamma(s)\} = \{\gamma(t)\} \Big) \bigg\}.
\end{split}
\end{equation*}
It is fairly easy to see that for every $\gamma$ $N'_{n}(\gamma)$ is finite for every $\gamma$: indeed if $s_n,t_n \in [0,L(\gamma)]$ are a sequence of points in the definition of $N'_n(\gamma)$, then it is clear that $|s_n-s_m| \vee |t_n - t_m| \geq 2^{-n}$. Hence $N'(\gamma)$ is countable and $N'$ is $\mathcal L^1 \times \pi$-negligible.

For a given $f \in S^1(X)$, consider the conegligible set
\begin{equation*}
M = \Big\{ (t,\gamma) \notin N': \dot{|\gamma|}(t) \ \text{exists and } f \circ \gamma \ \text{differentiable in $t$} \Big\}.
\end{equation*}
Note that if $\gamma(s) = \gamma(t)$, $s,t \in M(\gamma)$, then there are sequences $s_n \to s,t_n \to t$, $s_n,t_n \not= s,t$ respectively such that $\gamma(s_n) = \gamma(t_n)$: hence by the formula \eqref{Equa:simple_plan_ct} we deduce that 
\begin{equation*}
\frac{d}{dt} f \circ \gamma(t) = \frac{d}{dt} f \circ \gamma(s).
\end{equation*}
The complement $N = M^c$ satisfies the statement.
%
%
%
%
%
\end{proof}

In particular we deduce that

\begin{corollary}
\label{Cor:Lambda_pi_curve}
For $\pi$-a.e. $\gamma$ the function $\Lambda_\gamma$ depends only on $\gamma(t)$ for $\mathcal L^1$-a.e. $t$, i.e. can be written as a function on $\gamma([0,1])$ defined $\mathcal H^1$-a.e. or equivalently it is a measurable function for the pull-back $\sigma$-algebra $e^{-1}(\gamma)(\mathcal B(X))$.
\end{corollary}


%
%
%
%
%
%
%

\subsection{The integer circuitation condition}
\label{Ss:integer_rot_cond}

Let $v \in L^1(T^* X)$ and $\Lambda_\gamma(t) = \Lambda(t,v) \in L^1(\mathcal L^1 \times \pi)$ be the corresponding $L^1$-function given by \eqref{Equa:simple_plan_ct}: being $v$ fixed from now on, for shortness we will not write the dependence w.r.t. $v$. We restate the assumption of integer circuitation \eqref{Equa:integer_circ_32} here below.

\begin{assumption}
\label{Def:direc_fidl}
For all test plans $\pi$ concentrated on $\Gamma_c$ it holds
\begin{equation}
\label{Equa:inter_circ}
\int_0^1 \Lambda_\gamma(t) \dot{|\gamma|}(t) dt \in 2 \pi \Z \quad \pi\text{-a.e.} \ \gamma,
\end{equation}
i.e. $v$ (or $\Lambda$) has \emph{integer circuitation}. 
\end{assumption}

We define
\begin{equation*}
\gls{InteCirc} = \bigg\{ \gamma \in \Gamma_c, \int_0^1 \Lambda_\gamma(t) \dot{|\gamma|}(t) dt \in 2 \pi \Z \bigg\},
\end{equation*}
so that we can rephrase the above assumption as $\mathcal I_c$ has full $\pi$-measure. Note that as far as we know $\mathcal I_c$ depends on $\pi$, but this should not create confusion being $\pi$ fixed in the construction of the function $w$ satisfying \eqref{Equa:Diffe:relw} (see Section \ref{S:constr_w}). 

An elementary example of an integer circuitation cotangent vector field is given by
$$
v = - i \frac{Df}{f},
$$
where $f \in L^0(\mu,\mathbb S^1)$ with $|Df| \in L^1$: in this case
\begin{equation*}
\Lambda_\gamma(t) = - \frac{i}{f} \frac{d}{dt} f \circ \gamma(t) \quad (\dot{|\gamma|} = 1),
\end{equation*}
with $\pi$-a.e. $\gamma$, with $\pi$ test plan. 

\begin{remark}
\label{Rem:free_circ}
Another possible condition is to ask that
\begin{equation*}
\int_0^1 \Lambda_\gamma(t) \dot{|\gamma|}(t) dt = 0 \quad \pi\text{-a.e.} \ \gamma.
\end{equation*}
In this case we are looking for functions whose differential along $\gamma$ is $\Lambda_\gamma$, i.e. $v = Df$ is exact: we will say that every admissible $\pi$ has \emph{$0$ circuitation} or it is \emph{circuitation free}. The analysis of this situation is completely similar to the case we are considering in the following, so the same constructions apply.
\end{remark}

For completeness we show that we can restrict the set of test plans to a more tame one. 

\begin{proposition}
\label{Prop:gene_test_plan}
Assume that Assumption \ref{Def:direc_fidl} holds only for test plans which satisfy
\begin{equation*}
L(\gamma) \in L^\infty(\pi), \quad e(t)_\sharp \pi \leq C \mu,
\end{equation*} 
Then Assumption \ref{Def:direc_fidl} holds for all test plans.
\end{proposition}

\begin{proof}
Let $\pi \in \adm(\mu)$ be a test plan supported on $\Gamma_c$ as in Definition \ref{Def:admissi}.

Since Assumption \ref{Def:direc_fidl} is invariant for countable unions of test plans, we can assume that $\pi$ is concentrated on curves with length $L(\gamma) \in (\bar L,\bar L + 1)$, $\bar L$ fixed, and we thus consider the pameterization
\begin{equation*}
L(\gamma \llcorner_{[0,t]}) = L(\gamma) t.
\end{equation*}
In this case Condition \eqref{Equa:curves_pi} implies
\begin{equation*}
\int_0^1 e(t)_\sharp \pi \leq \frac{C}{\bar L} \mu = C' \mu.
\end{equation*}

Hence for all $\delta t > 0$ the measure $\pi'$ defined by
\begin{equation*}
\pi' = \fint_0^{\delta} \big[ (\gamma^{\epsilon})_\sharp \pi \big] d\epsilon, \quad \gamma^{\epsilon}(t) = \gamma(\epsilon + t \mod L(\gamma)),
\end{equation*}
is supported on $\Gamma_c$ and satisfies $\bar L \leq L(\gamma) \leq \bar L + 1$ and
\begin{equation*}
e(t)_\sharp \pi' = \fint_{t}^{\delta + t} \big[ e(s)_\sharp \pi \big] ds \leq \frac{C'}{\delta} \mu,
\end{equation*}
where for simplicity we assume that $0 \leq t \leq \bar L - \delta$. The assumptions on the proposition thus imply that Assumption \ref{Def:direc_fidl} holds for $\pi'$-a.e. $\gamma$, and by Corollary \ref{Cor:Lambda_pi_curve}
\begin{equation*}
\Lambda_{\gamma^\epsilon}(t) = \Lambda_\gamma(t+\epsilon \mod L(\gamma)),
\end{equation*}
so that we conclude that $\pi$ satisfies Assumption \ref{Def:direc_fidl}.
\end{proof}

\subsection{Intersection properties of curves}
\label{Ss:inter_prp}

By definition, any test plan $\pi$ supported on $\Gamma_c$ is concentrated on the set of curves such that
\begin{equation*}
\int_0^{L(\gamma)} \Lambda_\gamma(t) dt \in 2\pi \Z.
\end{equation*}
However a priori nothing is said about properties of intersecting curve: for example, given two curves $\gamma,\gamma'$ such that
\begin{equation}
\label{Equa:intersect}
\gamma(s) = \gamma'(s'), \ \gamma(t) = \gamma'(t'), \quad s \leq t, s' \leq t',
\end{equation}
one can construct a closed curve $\gamma''$ by setting
\begin{equation}
\label{Equa:ciclgam}
\gamma''(\tau) = \begin{cases}
R_{s,t}(\gamma)(2\tau) & \tau \in [0,1/2], \\
R_{1-t',1-s'}(R(\gamma'))(2\tau-1) & \tau \in (1/2,1],
\end{cases}
\end{equation}
i.e. moving from $\gamma(s)$ to $\gamma(t) = \gamma'(t')$ along $\gamma(\tau)$ and then returning to $\gamma'(s') = \gamma(s)$ along $\gamma(-\tau)$. We will use the above construction (i.e. a new closed curve obtained by piecing together parts of other curves) to show that we can also require
\begin{equation*}
\int_0^1 \Lambda(\gamma'')(\tau) \dot{|\gamma''|}(\tau) d\tau \in 2\pi \Z
\end{equation*}
by removing a $\pi$-negligible set of trajectories, where $\gamma''$ is the curve constructed above. 

We first start with self-intersecting curves.

\begin{proposition}
\label{Prop:self_inters}
Let $\Lambda$ be of integer circuitation. Every $\pi \in \adm(\mu)$ supported on $\Gamma_c$ is concentrated on a set of closed curves \gls{Api} such that if $\gamma(s) = \gamma(t)$ then $R_{s,t}(\gamma)$ has integer circuitation.
\end{proposition}

\begin{proof}
The proof is given in 3 steps. 

{\it Step 1: construction of a countable Borel set of loops.}
The following argument is the same as in \cite[Proposition 5.2.7]{Sri:courseborel}: we give it for completeness. The set
\begin{equation*}
F = \big\{ (\gamma,s,t) : \gamma(s) = \gamma(t) \big\} \subset \Gamma_c \times [0,1]^2
\end{equation*}
is closed, and moreover
\begin{equation*}
\Big\{ \gamma : \gamma(s) = \gamma(t) \ \text{for some} \ (s,t) \in \bar B^2_r((\bar s,\bar t)) \Big\}
\end{equation*}
is closed in $\Gamma$. Hence $F : \Gamma_c \to [0,1]^2$ is a Borel measurable multifunction with $F(\gamma)$ closed. For every $\bar B^2_{2^{-k}}((\bar s_m,\bar t_m))$, $s_m,t_m \in \Q$ and $k \in \N$, there is a Borel selection \cite[Theorem 5.2.1]{Sri:courseborel}
\[
L_{k,m} : A_{k,m} = F^{-1}(\bar B^2_{2^{-k}}((\bar s_m,\bar t_m)))  \to [0,1]^2, \quad L_{k,m}(\gamma) = (s_{k,m},t_{k,m}) \in F(\gamma) \cap \bar B^2_{2^{-k}}((\bar s_m,\bar t_m)),
\]
By construction, the graphs of the functions $L_{k,m}$ are dense in $F(\gamma)$ for each $\gamma$, i.e.
\begin{equation*}
F(\gamma) = \clos \big\{ L_{k,m}(\gamma), \gamma \in A_{k,m} \big\}.
\end{equation*}
By reordering, we will denote the functions $L_{m,k}$ and domains $A_{m,k}$ with $L_n$, $A_n$ respectively.
%
%

{\it Step 2: negligibility of the set of bad loops.}
%
If $\pi \in \mathcal P(\Gamma_c)$ is a test plan, for each map $L_n(\gamma) = (s_n,t_n)$ define the transport plan
\begin{equation*}
\pi_n = (R_{s_n,t_n})_\sharp (\pi \llcorner_{A_n}), \quad A_n = \Dom(L_n).
\end{equation*}
Being a restriction, clearly $\pi_n \in \adm(\mu)$. In particular, it follows by Assumption \ref{Def:direc_fidl} that
\begin{equation*}
\pi_n \big( \{ \gamma \notin \mathcal I_c \} \big) = \pi \big( \big\{ \gamma \in A_n : R_{s_n,t_n}(\gamma) \notin \mathcal I_c \big\} \big) = 0.  
\end{equation*}
By $\sigma$-additivity
\begin{equation*}
\pi \big( \big\{ \gamma \in A_n : \exists n \big( R_{s_n,t_n}(\gamma) \notin \mathcal I_c \big) \big\} \big) = 0.
\end{equation*}

{\it Step 3: conclusion.}
The previous point gives that $\pi \in \adm(\mu)$ is concentrated on a set of curves $A_\pi$ such that for all $n \in \N$ it holds $R_{s_n,t_n}(\gamma) \in \mathcal I_c$. Let now $(s_n,t_n),(\bar s,\bar t) \in F(\gamma)$, with $(s_n,t_n) \to (\bar s,\bar t)$. Then 
\begin{equation*}
2 \pi \Z \ni \int_{s_n}^{t_n} \Lambda_\gamma(\tau) \dot{|\gamma|}(\tau) d\tau \to \int_{\bar s}^{\bar t} \Lambda_\gamma(\tau) \dot{|\gamma|}(\tau) d\tau,
\end{equation*}
which gives that every subcurve of $\gamma \in A_\pi$ belongs to $\mathcal I_c$, i.e. it is of integer circuitation.
%
%
%
%
\end{proof}

As a consequence of the previous proposition, the next lemma shows that we can construct the function $w$ along $\pi$-a.e. curve.

\begin{lemma}
\label{Lem:simple_zero}
Let $\gamma$ be a Lipschitz curve such that if $\gamma(s) = \gamma(t)$ then $R_{s,t} \gamma \in \mathcal I_c$. Then there exists a function $w_\gamma : \gamma([0,1]) \to \mathbb S^1$ such that
\begin{equation*}
\frac{d}{dt} w_\gamma(\gamma(t)) = i \Lambda_\gamma(t) \dot{|\gamma|}(t) w(\gamma(t)).
\end{equation*}
\end{lemma}

\begin{proof}
Define the function
\begin{equation*}
w_\gamma(t) = e^{i \int_0^t \Lambda_\gamma \dot{|\gamma|}  \mathcal L^1}.
\end{equation*}
If $\gamma(t) = \gamma(s)$, then by the assumption $R_{s,t} \gamma \in \mathcal I_c$ it holds 
\begin{equation*}
w_\gamma(t) = w_\gamma(s) e^{i \int_s^t \Lambda_\gamma \dot{|\gamma|} \mathcal L^1} = w_\gamma(s),
\end{equation*}
so that the function $w_\gamma$ can be actually defined on $\Graph(\gamma)$ by
\begin{equation*}
w_\gamma(x) = w_\gamma(t), \quad x = \gamma(t). \qedhere
\end{equation*}
\end{proof}

We next generalize Proposition \ref{Prop:self_inters} to closed subcurves of the set of curves where $\pi$ is concentrated: given a set $A \subset \Gamma_c$ define
\begin{equation*}
\gls{An} = \bigg\{ \gamma \in \Gamma_c : \exists \gamma_1,\dots,\gamma_n \in A \bigg( \Graph \gamma \subset \bigcup_i \Graph \gamma_i \bigg) \bigg\}.
\end{equation*}

\begin{theorem}
\label{theo:neglinters}
Let $\Lambda$ be of integer circuitation. Every $\pi \in \adm(\mu)$ is concentrated on a set of curves $A_\pi$ such that $A_\pi(n) \subset \mathcal I_c$ for all $n$, i.e. every subloop of $A_\pi$ belongs to $\mathcal I_c$.
\end{theorem}

\begin{proof}
The proof is given in two steps. The first one illustrates the procedure in a simple case, while the second is the general case.

{\it Step 1: the case of two curves.}
Consider the set $(\Gamma_c)^2 = \Gamma_c \times \Gamma_c$ and its closed subset
\begin{equation*}
\mathcal C(2) = \Big\{ (\gamma_1,\gamma_2) \in (\Gamma_c)^2 : \gamma_1([0,1]) \cap \gamma_2([0,1]) \not= \emptyset \Big\}.
\end{equation*}
In particular the multifunction $F(\gamma_1,\gamma_2)$ with graph
\begin{equation*}
F = \Big\{ (\gamma_1,\gamma_2,t_1,t_2) \in \mathcal C(2) \times [0,1]^2 : \gamma_1(t_1) = \gamma_2(t_2) \Big\}
\end{equation*}
is Borel measurable and $F(\gamma_1,\gamma_2)$ is closed, so that by \cite[Theorem 5.2.1]{Sri:courseborel} there is a Borel selection 
\begin{equation*}
L : \mathcal C(2) \to [0,1]^2, \quad L(\gamma_1,\gamma_2) = (t_1,t_2) \ \text{such that $\gamma_1(t_1) = \gamma_2(t_2)$}. 
\end{equation*}
We thus can define the curve
\begin{equation}
\label{Equa:mergiloop}
\gamma'(\gamma_1,\gamma_2;t) = \begin{cases}
\gamma_1(2t) & t \in [0,t_1/2] \\
\gamma_2(t_2 + 2t-t_1 \mod 1) & t \in (t_1/2,t_1/2+1/2), \\
\gamma_1(2t - 1) & t \in [t_1/2+1/2,1].
\end{cases}
\end{equation}
Clearly the map $(\gamma_1,\gamma_2) \to \gamma'(\gamma_1,\gamma_2)$ is Borel. Up to reparametrization, the curve $\gamma'(\gamma_1,\gamma_2)$ is the closed curve obtained by moving along the first curve and then along the second one.

Consider now the set of positive measures
\begin{equation*}
\Pi^\leq(\pi,\pi) = \Big\{ \xi \in \mathcal M^+(\mathcal C(2)) : (\mathtt p_1)_\sharp \xi \leq \pi, (\mathtt p_2)_\sharp \xi \leq \pi \Big\},
\end{equation*}
where $\pi \in \adm(\mu)$ is the test plan in the statement. We can define the new test plan
\begin{equation*}
\pi' = (\gamma')_\sharp \xi, \quad \xi \in \Pi^\leq(\pi,\pi).
\end{equation*}
It is fairly easy to see that $\pi' \in \adm(\mu)$ because
\begin{equation*}
\begin{split}
\int \bigg[ \int_0^1 \phi(\gamma(t)) \dot{|\gamma|}(t) dt \bigg] \pi'(d\gamma) &= \sum_{i=1,2} \int \bigg[ \int_0^1 \phi(\gamma(t)) \dot{|\gamma|}(t) dt \bigg] (\mathtt p_i)_\sharp \xi(d\gamma) \\
&\leq 2 \int \bigg[ \int_0^1 \phi(\gamma(t)) \dot{|\gamma|}(t) dt \bigg] \pi(d\gamma), 
\end{split}
\end{equation*}
and since $\gamma' \in \Gamma_c$ then $\pi'$ is supported on $\Gamma_c$.

Proposition \ref{Prop:self_inters} implies that the set
\begin{equation*}
\tilde{\mathcal C}(2) = \Big\{ (\gamma_1,\gamma_2) \in \mathcal C(2) : \gamma'(\gamma_1,\gamma_2) \ \text{has a self-intersection with non integer circuitation} \Big\}
\end{equation*}
is negligible for every $\xi \in \Pi^\leq(\pi,\pi)$. We recall now the following result \cite[Theorem 2.19 and Proposition 3.5]{kel:duality}: if $B$ is Borel (or analytic)
\begin{equation*}
\sup_{\xi \in \Pi^{\leq}(\pi,\pi)} \xi(B) = \min \Big\{ \pi(B_1) + \pi(B_2):\, B \subset (B_1 \times \Gamma) \cup (\Gamma \times B_2) \Big\}.
\end{equation*}
In our case, we have observed that the l.h.s. is $0$ when $B = \tilde{\mathcal C}(2)$, and then there are set $B_1$, $B_2$ negligible w.r.t. $\pi$ such that for every $\gamma_1,\gamma_2 \in \Gamma_c \setminus (B_1 \cup B_2)$ every subloop of the curve $\gamma'(\gamma_1,\gamma_2)$ has integer circuitation.

This proves the theorem for $n=2$.

{\it Step 2: the general case}.
Define the closed set
\begin{equation*}
\gls{calCn} = \bigg\{ (\gamma_1,\dots,\gamma_n) \in (\Gamma_c)^n : \bigcup_i \Graph \gamma_i \ \text{is connected} \bigg\}.
\end{equation*}
Starting with $\gamma_1$, as in the previous step we can join it to some curve $\gamma_i$, $i \not= 1$, by formula \eqref{Equa:mergiloop}: call this curve $\tilde \gamma_2$. Using the assumption that $\cup_i \Graph \gamma_i$ is connected, we can repeat the process by adding to $\tilde \gamma_2$ a new curve $\gamma_j$, $j \not= 1,i$, as in \eqref{Equa:mergiloop}, obtaining a new curve $\tilde \gamma_3$. Proceeding by adding curves, after $n$ steps we obtain a single curve $\tilde \gamma$: it is fairly easy to see that since each merging operation is Borel, so the map $(\gamma_1,\dots,\gamma_n) \to \tilde \gamma(\gamma_1,\dots,\gamma_n)$ is Borel too.

As in the previous step, for every
\begin{equation*}
\xi \in \Pi^\leq \big( \underbrace{\pi,\dots,\pi}_{n \ \text{times}} \big) = \big\{ \xi \in \mathcal M^+(\mathcal C(n)), (\mathtt p_i)_\sharp \xi \leq \pi \big\}
\end{equation*}
we can construct the test plan
\begin{equation*}
\pi' = \big( \tilde \gamma(\gamma_1,\dots,\gamma_n) \big)_\sharp \xi,
\end{equation*}
so that we deduce that the set
\begin{equation*}
\tilde{\mathcal C}(n) = \Big\{ (\gamma_1,\dots,\gamma_n) \in \mathcal C(n) : \gamma' \ \text{has a self intersection with non integer circuitation} \Big\}
\end{equation*}
is negligible for every $\xi \in \Pi^\leq(\pi,\dots,\pi)$. We now use \cite[Lemma 1.8, Theorems 2.14 and 2.21]{kel:duality}: they state that if $B$ is Borel (or analytic) then
\begin{equation*}
\sup_{\xi \in \Pi^{\leq}(\pi,\dots,\pi)} \xi(B) = \min \bigg\{ \sum_{i=1}^n \int f_i \pi:\, \ind_B(\gamma_1,\dots,\gamma_n) \leq \sum_{i=1}^n f_i(\gamma_i), 0 \leq f_i \leq 1 \bigg\}.
\end{equation*}
In our case the left side is $0$, so that we can take the sets $B_i = \{f_i > 0\}$ to deduce that there are $\pi$-negligible sets $B_i$ such that
\begin{equation*}
\tilde{\mathcal C}(n) \subset \bigcup_i (\mathtt p_i)^{-1}(B_i), \quad \text{where} \ \mathtt p_i(\gamma,\dots,\gamma_n) = \gamma_i.
\end{equation*}
In particular in the complement of $\tilde B_n = \cup_i B_i$ it holds that every subloop of $\tilde \gamma(\gamma_1,\dots,\gamma_n)$ belongs to $\mathcal I_c$. Removing the negligible set $\cup_n \tilde B_n$ we obtain the set $A_\pi$ of the statement.
\end{proof}

Given a sequence of test plan $\pi_n$, $n \in \N$, it is immediate to see that
\begin{equation}
\label{Equa:countabl_pi}
\frac{1}{\sum_n \frac{2^{-n}}{C_n}} \sum_n \frac{2^{-n}}{C_n} \pi_n \in \adm(\mu), 
\end{equation}
where $C_n$ is the constant for $\pi_n$ in \eqref{Equa:curves_pi} (the constant in front the sum is for normalization). Hence the previous theorem gives the following statement.

\begin{corollary}
\label{Cor:count_m}
If $\{\pi_n\}_{n \in \N} \subset \adm(\mu)$ is a countable family of $\mu$-test plans, then there exist Borel sets $A_{\pi_n}$ such that $\pi_n(A_{\pi_n}) = 1$ and it holds
\begin{equation*}
\bigg\{ \gamma \in \Gamma_c : \exists \gamma_1 \in A_{\pi_1},\dots,\gamma_n \in A_{\pi_n} \bigg( \Graph \gamma \subset \bigcup_i \Graph \gamma_i \bigg) \bigg\} \subset \mathcal I_c.
\end{equation*}
\end{corollary}

\section{A simple case: $x \in \R^2$, $\mu = f^2 \mathcal L^2/\|f\|_2^2$ with $f \in W^{1,2}$}
\label{S:case_R2}

As a preliminary analysis, we perform the construction of a function $w : \R^2 \mapsto \mathbb S^1$ such that
\begin{equation*}
D(fw) = Df w + f v,
\end{equation*}
where $v \in \R^d$ is a (cotangent) vector field defined $f$-a.e. having integer circuitation, and $f \in W^{1,2}(\R^2,\mathcal L^d)$: the integer circuitation assumption is
\begin{equation*}
\int_0^1 v(\gamma(t))\cdot \dot \gamma(t) dt \in 2\pi\Z
\end{equation*}
for $\pi$-a.e. $\gamma$, with $\pi$-test plan concentrated on $\Gamma_c(\R^2)$ for the measure $\mu = f^2 \mathcal L^2/\|f\|_2^2$. With some additional technicalities, a similar analysis can be adapted to the $d$-dimensional case: the advantage in the $2d$-case is that the construction of the test plans for the measure $\mu = f^2 \mathcal L^2$ is particularly easy.



Consider a Lebesgue point for $f$ and $Df$: let us assume for definiteness
\begin{equation}
\label{Equa:Lebesgue_f_Df_1}
\fint_{B^2_r(0)} |f - 1|^2 \mathcal L^d, \fint_{B^2_r(0)} |Df - 1|^2 < \epsilon^2, \quad \epsilon = o(r).
\end{equation}
By standard density argument, for the directions $\mathbf e_1,\mathbf e_2$ along the coordinate axis there are segments
$$
\ell_{k,y} = \big\{ x \in \R^2 : \mathbf e_k \cdot x \in [-r,r], \mathtt p_{\mathbf e_k^\perp} x = y \big\}, \quad y \in K_k \subset [-r,r] \ \text{compact}, 
$$
such that
\begin{equation*}
\mathcal L^{1}(K_k) \geq r, \quad \int_{\ell_{k,y}} |f-1|^2 \mathcal H^1, \int_{\ell_{k,y}} |\partial_{x_k}f - 1|^2 \mathcal H^1 < C \epsilon^2 r.
\end{equation*}
In particular it follows that $|f| \geq 1/2$ $\mathcal H^1$-a.e. on all the segments $\{\ell_{k,y}, y \in K_k, k = 1,2\}$.

Consider the map $\R^4 \ni (y,y') = (y_1,y_2,y_1',y_2') \mapsto \gamma_{(y,y')}(t) \in \Gamma_c$ where
\begin{equation}
\label{Equa:curve_gamma_yy'}
\gamma_{(y,y')} = \overrightarrow{[y,(y_1',y_2)]} \cup \overrightarrow{[(y_1',y_2),y']} \cup \overrightarrow{[y',(y_1,y_2')]} \cup \overrightarrow{[(y_1,y_2'),y]},
\end{equation}
where $\overrightarrow{[x,y]}$ is the geodesic connecting $x,y$: $\gamma_{(y,y')}$ is the oriented boundary of the rectangle with opposite corners $y,y'$.

\begin{lemma}
\label{Lem:test_R2}
For all $g \in L^\infty(\R^4,\mathcal L^4)$, $\|g\|_1 = 1$ with bounded support, the measure $\pi = (\gamma_{(y,y')})_\sharp (g\mathcal L^4)$ is a test plan in $(\R^2,\mathcal L^2)$.
\end{lemma}

\begin{proof}
Since $\gamma_{(y,y')}$ is the union of 4 curves, it is enough to prove the statement for a single part, for example $\overrightarrow{[y,(y_1',y_2)]}$: in this case
\begin{equation*}
\begin{split}
\int \phi \big( (1-t)y_1 + ty_1',y_2 \big) g(y,y') dydy' &= \int \phi(z_1,y_2) g(z,y') ((1-t)^2 + t^2) dzdy' \\
&\leq C \|g\|_\infty \big( \diam(\supp g) \big)^2 \int \phi(z_1,y_2) dz_1 dy_2,
\end{split}
\end{equation*}
where
\begin{equation*}
\left( \begin{array}{c}
z_1 \\ z_1' \end{array} \right) = \left[ \begin{array}{cc} 1-t & t \\ -t & 1-t \end{array} \right] \left( \begin{array}{c}
y_1 \\ y_1' \end{array} \right),
\end{equation*}
which shows that the evaluation of $\pi$ is $\leq C(g) \mathcal L^2$.
\end{proof}

Using the same ideas of the first part of the proof of Theorem \ref{theo:neglinters} applied to the measures
$$
\xi \in \Pi^\leq \big( \mathcal L^2 \llcorner_{K_1 \times K_2}, \mathcal L^2 \llcorner_{K_1 \times K_2} \big),
$$
we deduce the following corollary.

\begin{corollary}
\label{Cor:every_traj_good}
There is a $\mathcal L^2$-conegligible set $I_0 \subset K_1 \times K_2$ such that for every $y,y' \in I_0$ the curve $\gamma_{(y,y')}$ has integer circuitation.
\end{corollary}

Using Fubini, we obtain that for $\mathcal L^2$-a.e. $y = (y_1,y_2) \in I_0$ the set
\begin{equation}
\label{Equa:full_set}
\big\{ y' \in I_0 : p_2(y') = y_2 \big\} \subset K_2,
\end{equation}
has full measure. Define for one of these $y$
\begin{equation*}
w(y,y') = e^{i \int_{y_1}^{y_1'} v_1(t,y_2) dt + i \int_{y_2}^{y_2'} v_2(y_1',t) dt}, \quad y' \in K_1 \times K_2,
\end{equation*}
i.e. the integral along the curve $\overrightarrow{[y,(y_1',y_2)]} \cup \overrightarrow{[(y_1',y_2),y']}$. The function $w(y,y')$ is defined on $\mathcal L^2$-a.e. point of $K_1 \times K_2$. 

\begin{lemma}
\label{Lem:w_funct_good}
For every $(y',y'') \in (K_1 \times K_2)^2 \setminus N$, $\mathcal L^4(N) = 0$, it holds
\begin{equation*}
w(y,y'') = w(y,y') e^{i \int_{y'_1}^{y_1''} v_1(t,y_2') dt + i \int_{y'_2}^{y_2''} v_2(y_1'',t) dt}. 
\end{equation*}
\end{lemma}

\begin{proof}
The proof is a consequence that the statement holds for $y'' \in I_0$ such that $(y''_1,y_2) \in I_0$ because setting $y''' = (y''_1,y_2)$
\begin{equation*}
\begin{split}
& \int_{y_1}^{y_1'} v_1(t,y_2) dt + \int_{y_2}^{y_2'} v_2(y_1',t) dt 
+ \int_{y'_1}^{y_1''} v_1(t,y_2') dt + \int_{y'_2}^{y_2''} v_2(y_1'',t) dt \\
& \quad = \int_{y_1}^{y_1''} v_1(t,y_2) dt + \int_{y_2}^{y_2''} v_2(y_1'',t) dt 
+ \int_{\gamma_{(y',y''')}} v_1 \cdot \dot \gamma_{(y',y''')} \\\
& \quad \in \int_{y_1}^{y_1''} v_1(t,y_2) dt + \int_{y_2}^{y_2''} v_2(y_1'',t) dt 
+ 2\pi \Z, 
\end{split}
\end{equation*}
and the set
$$
\big\{ y' \in I_0, (y'_1,y_2) \in I_0 \big\} \times \big\{ y'' \in I_0, (y''_1,y_2) \in I_0 \big\}
$$
has full measure due to the assumption on \eqref{Equa:full_set} and on $I_0$.
%
%
\end{proof}

By repeating the construction, we obtain the following corollary.

\begin{corollary}
\label{Cor:count_w}
The is a countable family of sets $K_{1,n} \times K_{2,n}$ and of functions $w_n$ such that
\begin{equation*}
\bigcup_{n \in \N} K_{1,n} \times K_{2,n} \quad \text{has full $|f|^2 \mathcal L^2$-measure}
\end{equation*}
and $w_n$ satisfies Lemma \ref{Lem:w_funct_good} in every $K_{1,n} \times K_{2,n}$.
\end{corollary}

Note that we do not require the sets $K_{1,n} \times K_{2,n}$ to be disjoint.

\begin{proof}
Since by restricting the balls we can require that
\begin{equation*}
\frac{\mathcal L^2(K_{1,n} \times K_{2,n})}{(2r)^2} \to 1, 
\end{equation*}
a standard covering argument as in the proof of \cite[Theorem 5.5.1]{bog:measure} gives the statement.
\end{proof}

The above lemma defines the function $w$ in a set of positive Lebesgue measure, with the property that every two points are connected with a path where $|f| \geq 1/2$. This is the hardest step, since the remaining part needs only to show that this $w$ is compatible with all test plans and it can be extended to the whole $\{f \not= 0\}$.


Consider a test plan $\pi$ such that
\begin{equation*}
e(0)_\sharp \pi \leq C \mathcal L^2 \llcorner_{K_{1,i} \times K_{2,i}}, \quad e(1)_\sharp \pi \leq C \mathcal L^2 \llcorner_{K_{1,j} \times K_{2,j}}.
\end{equation*}
For every $\gamma,\gamma'$ such that
\begin{equation*}
\gamma(0),\gamma'(0) \in K_{1,i} \times K_{2,i}, \quad \gamma(1),\gamma'(1) \in K_{1,j} \times K_{2,j},
\end{equation*}
let $\hat \gamma_{\gamma,\gamma'}$ be the closed curve obtained by joining $\gamma(0),\gamma'(0)$ and $\gamma(1),\gamma'(1)$ with the curves used in Lemma \ref{Lem:w_funct_good} to compute $w$. Define
\begin{equation*}
\tilde \pi = (\hat \gamma_{\gamma,\gamma})_\sharp \big( \pi(d\gamma) \times \pi(d\gamma') \big),
\end{equation*}
which is a transference plan due to Lemma \ref{Lem:test_R2}. The condition of integer circuitation yields that by removing a $\tilde \pi$-negligible set of trajectories the curves $\hat \gamma_{\gamma,\gamma}$ have integer circuitation. For all these curves it follows that
\begin{equation}
\label{Equa:ratio_1}
\begin{split}
\frac{w_j(\gamma(1))}{w_i(\gamma(0)) e^{i \int_0^1 v \cdot \dot \gamma \mathcal L^1}} &= \frac{w_j(\gamma'(1)) e^{i \int v \cdot \dot{\overrightarrow{[\gamma'(1),\gamma(1)]}}}}{w_i(\gamma'(0)) e^{i \int v \cdot \dot{\overrightarrow{[\gamma'(0),\gamma(0)]})} + i \int_0^1 v \cdot \dot \gamma \mathcal L^1}} = \frac{w_j(\gamma'(1))}{w_i(\gamma'(0)) e^{i \int_0^1 v \cdot \dot \gamma' \mathcal L^1}}. 
\end{split}
\end{equation}
The above computation yields the following lemma.

\begin{lemma}
\label{Lem:single_constat}
There exists a constant $c_{ij}$ such that for every test plan $\pi$ with
$$
e(0)_\sharp \pi \leq C \mathcal L^2 \llcorner_{K_{1,i} \times K_{2,i}}, \quad e(1)_\sharp \pi \leq C \mathcal L^2 \llcorner_{K_{1,j} \times K_{2,j}},
$$
it holds for $\pi$-a.e. $\gamma$
\begin{equation*}
c_{ij} w_j(\gamma(1)) = w_i(\gamma(0)) e^{i \int v \cdot \dot  \gamma}.
\end{equation*}
\end{lemma}

\begin{proof}
For $\pi$-a.e. $\gamma$ the set $\gamma'$ such that $\hat \gamma_{\gamma,\gamma'}$ is of integer circuitation has full $\pi$-measure. For every such a $\gamma$, define $c_{ij}$ as the inverse of the ratio of \eqref{Equa:ratio_1}:
\begin{equation*}
c_{ij} = \frac{w_i(\gamma(0)) e^{i \int_0^1 v \cdot \dot \gamma \mathcal L^1}}{w_j(\gamma(1))}.
\end{equation*}
Since by \eqref{Equa:ratio_1} the ratio is the same for $\pi$-a.e. $\gamma$, the proof is complete.
\end{proof}

%
%
%
%
%

Starting from the function $w_1$ defined in $K_{1,1} \times K_{2,1}$, we can then find constants $c_{1j}$ such that the above lemma holds for all sets $K_{1,j} \times K_{2,j}$ for which there is a test plan connecting $K_{1,1} \times K_{2,1}$ to $K_{1,j} \times K_{2,j}$. If at this step we cover $\{|f| > 0\}$ up to a $\mathcal L^2$-negligible set, we are done, otherwise we take the next $K_{1,i} \times K_{2,i}$ and connect to some others subsets $K_{1,j'} \times K_{2,j'}$: clearly these sets cannot belong to the previous family, otherwise there is a test plan from $K_{1,1} \times K_{2,1}$ to $K_{1,j'} \times K_{2,j'}$ and then to $K_{1,i} \times K_{2,i}$.

We this conclude with the following statement.

\begin{proposition}
\label{Prop:w_good}
There exists a function $w$ such that for all test plans $\pi$ w.r.t. the measure $|f|^2 \mathcal L^2$ it holds
\begin{equation*}
w(\gamma(1)) = w(\gamma(0)) e^{i \int v \cdot \dot \gamma} \quad \text{$\pi$-a.e. $\gamma$}.
\end{equation*}
Moreover all these functions depend at most on countably many multiplicative constants $c_j$.
\end{proposition}

\begin{proof}
We are left in proving the second part, which follows from the following observations:
\begin{enumerate}
\item every solution $\tilde w$ satisfies
\begin{equation*}
\frac{\tilde w}{w_j} = \ \text{constant in every $K_{1,j} \times K_{2,j}$}:
\end{equation*}
indeed it is constant on the curves used in Lemma \ref{Lem:w_funct_good};
\item the relation $R$
\begin{equation*}
(K_{1,j} \times K_{2,j}) R (K_{1,i} \times K_{2,i}) \quad \Longleftrightarrow \quad \exists \pi \ \text{plan connecting them as in Lemma \ref{Lem:single_constat}}
\end{equation*}
is an equivalence relation;
\item the construction of $w$ presented before the statement of this proposition is defined up to a multiplicative constant in each equivalence class.
\end{enumerate}
This concludes the proof.
\end{proof}

Once the function $w$ is constructed, it follows from the chain rule in one dimension that for all test plans w.r.t. $|f|^2 \mathcal L^2$ it holds
\begin{equation*}
\frac{d}{dt} (fw) \circ \gamma = \frac{d f \circ \gamma}{dt} w + f \frac{d w \circ \gamma}{dt} = w Df \cdot \dot \gamma + i w f v \cdot \gamma, \end{equation*}
for $\pi$-a.e. $\gamma$. In particular, by elementary arguments, for $\mathcal L^1$-a.e. $y_2$ it holds $y_1 \mapsto (fw)(y_1,y_2)$ is a.c. and has derivative $w \partial_{x_1} f + i w v_1$, and this means that
\begin{equation*}
\|D(fw)\|_2 \leq \|Df\|_2 + \int |v|^2 |f|^2\mathcal L^2 < \infty,
\end{equation*}
Since $|fw| = |f|$ we thus conclude with the following result.


\begin{proposition}
\label{Prop:one_L2}
For every $w$ given by Proposition \ref{Prop:w_good} the function $fw$ is in $W^{1,2}(\R^2)$ and has differential $D(fw) = wDf + iwfv$.
\end{proposition}

%
%

\section{Construction of the minimal equivalence relation}
\label{S:constr_min_eq}

In this section we construct a decomposition $X = \cup_\alpha X_\alpha$ into disjoint sets with the following properties:
\begin{itemize}
\item each test plan is concentrated on a set of curves $\gamma$ such that $\gamma([0,1])$ belongs to a singe elements $Z_\alpha$ up to a $\mathcal H^1$-negligible est;
\item there exists a test plan $\bar \pi$ generating the partition $\{X_\alpha\}_\alpha$ up to a $\mu$-negligible set.
\end{itemize}
The first property means that every test plan move mass only inside the equivalence classes $X_\alpha$: this can be interpreted by saying that the $X_\alpha$ are the connected components from the point of view of the MMS $(X,\mu)$, in the sense that no mass transfer can occur between different equivalence classes.
For example, the MMS
\begin{equation}
\label{Equa:twoballs}
B_1((-2,0)) \cup B_1((2,0)) \cup [-1,1] \times \{0\} \subset \R^2
\end{equation}
is such that the segment connecting $[-1,1]$ has $\mathcal L^2$-measure $0$, so the no test plans $\pi$ can charge mass on it. Hence it is not seen by any Sobolev function $f \in W^{1,2}(X)$: from their point of view the two balls are completely disjoint. \\
The second property instead says that no smaller partitions can be constructed if they need to satisfy the first property: even by changing the set where the plan $\bar \pi$ is concentrated, the partition $\{X_\alpha\}_\alpha$ remains the same in a set of full $\mu$-measure. Returning to the example \eqref{Equa:twoballs}, it is elementary to see that the two balls cannot be further subdivided if we require the above first point to hold (recall Section \ref{S:case_R2} and see also Example \ref{Ex:torus_Riemann}).

In the first part of this section we take an abstract point of view as follow.
\begin{enumerate}
\item The (complete) $\sigma$-algebra \gls{calM} of $\mu$ is separable, i.e. countably generated: identifying elements of $\mathcal M$ up to negligible sets, we obtain a separable measure algebra which we still denote by $(\mathcal M,\mu)$ or $\mathcal M$ since the measure $\mu$ is fixed here. The function
\begin{equation}
\label{Equa:dist_meas_alg}
d(\Omega,\Omega') = \mu( \Omega \Delta \Omega' )
\end{equation}
is a distance on $\mathcal M$, which makes $\mathcal M$ a complete and separable metric space.
\item Each equivalence relation $E$ on $X$ generates the sub $\sigma$-algebra $\mathcal M^E$ made of all saturated measurable sets,
\begin{equation}
\label{Equa:gen-sigma}
\gls{calME} = \big\{ \Omega \in \mathcal M : E \Omega = \Omega \big\}.
\end{equation}
The corresponding measure algebra $\mathcal M^E$ is clearly a measure sub-algebra of $\mathcal M$. In particular it is separable and closed.
\end{enumerate}
The statements below can be deduced easily from the properties of measure algebras, see \cite[Chapter 3]{Fre:measuretheory3}.

%
%

Consider a family of measure sub-algebras $\{\mathcal M^\alpha\}_\alpha$, $\mathcal M^\alpha \subset \mathcal M$, closed under countable union: for every sequence $\mathcal M^{\alpha_i}$ there exists $\mathcal M^{\bar \alpha}$ such that
\begin{equation*}
\bigcup_i \mathcal M^{\alpha_i} \subset \mathcal M^{\bar \alpha}.
\end{equation*}
The following result shows that there is a maximal element $\mathcal M^{\bar \alpha}$.

\begin{proposition}
\label{Prop:max_measure_alg}
There exists $\mathcal M^{\bar \alpha} \in \{\mathcal M^\alpha\}_\alpha$ such that $\mathcal M^\alpha \subset \mathcal M^{\bar \alpha}$ for all $\alpha$.
\end{proposition}

The proof is just the observation that every subset of a separable metric space has a countable dense subset.

\begin{proof}
Consider the measure sub-algebra generated by the union of all $\mathcal M^\alpha$: it corresponds to the closure of the algebra generated by the algebra $\cup_\alpha \mathcal M^\alpha$ w.r.t. the distance \eqref{Equa:dist_meas_alg}. In particular, being countably generated, it coincides with the measure sub-algebra generated by countably many elements $\Omega_n \in \cup_\alpha \mathcal M^\alpha$.

Let the sub-algebras $\mathcal M^{\alpha_n}$ be such that $\Omega_n \in \mathcal M^{\alpha_n}$, and let $\mathcal M^{\bar \alpha}$ be measure sub-algebra containing all of them. This is the measure sub-algebra of the statement.
\end{proof}

A symmetric result considers the case where the family $\{\mathcal M^\alpha\}_\alpha$ is closed under countable intersection: for every sequence $\{\mathcal M^{\alpha_n}\}_n$ there exists $\mathcal M^{\bar \alpha}$ such that
\begin{equation}
\label{Equa:inter_clos}
\mathcal M^{\bar \alpha} \subset \bigcap_n \mathcal M^{\alpha_n}.
\end{equation}

\begin{proposition}
\label{Prop:min_measure_alg}
There exists $\mathcal M^{\bar \alpha} \in \{\mathcal M^\alpha\}_\alpha$ such that $\mathcal M^{\bar \alpha} \subset \mathcal M^{\alpha}$ for all $\alpha$.
\end{proposition}

The proof follows by observing that every union of open balls is obtained by considering only countably many balls.

\begin{proof}
First of all, for all
\begin{equation*}
\Omega \notin \bigcap_\alpha \mathcal M^\alpha
\end{equation*}
there is $\alpha_\Omega$ such that $\Omega \notin \mathcal M^{\alpha_\Omega}$, and then there exists $r_\Omega > 0$ such that the set
\begin{equation*}
B^{\mathcal M}_{R_\Omega}(\Omega) = \big\{ \Omega' : d(\Omega,\Omega') < r_\Omega \big\}
\end{equation*}
has empty intersection with $\mathcal M^{\alpha_\Omega}$.

By separability, we need only a countable family of balls $\{B^{\mathcal M}_{R_{\Omega_n}}(\Omega_n)\}_n$ to cover
\begin{equation*}
\mathcal M \setminus \bigcap_\alpha \mathcal M^\alpha.
\end{equation*}
Hence we deduce that if
$$
B^{\mathcal M}_{R_{\Omega_n}}(\Omega_n) \cap \mathcal M^{\alpha_n} = \emptyset,
$$
then
\begin{equation*}
\bigcap_\alpha \mathcal M^\alpha = \bigcap_n \mathcal M^{\alpha_n},
\end{equation*}
which is the statement by \eqref{Equa:inter_clos}.
\end{proof}

In terms of $\sigma$-algebras, the two statements above can be rewritten as follows. Consider a family of $\mu$-complete $\sigma$-algebras $\mathcal M^\alpha \subset \mathcal M$, where $\mathcal M$ is the $\mu$-completion of a separable $\sigma$-algebra.

\begin{corollary}
\label{Cor:sigma_min_max}
The following holds:
\begin{enumerate}
\item if $\{\mathcal M^\alpha\}_\alpha$ is closed under countable union, then there exists an element $\mathcal M^{\bar \alpha}$ such that
\begin{equation*}
\forall \alpha, \Omega \in \mathcal M^\alpha \, \exists \Omega' \in \mathcal M^{\bar \alpha} \big( \mu (\Omega \Delta \Omega') = 0 \big);
\end{equation*}
\item if $\{\mathcal M^\alpha\}_\alpha$ is closed under countable intersection, then there exists an element $\mathcal M^{\bar \alpha}$ such that
\begin{equation*}
\forall \alpha, \Omega \in \mathcal M^{\bar \alpha} \, \exists \Omega' \in \mathcal M^{\alpha} \big( \mu (\Omega \Delta \Omega') = 0 \big).
\end{equation*}
\end{enumerate}
\end{corollary}

The same statement can be rewritten in terms of equivalence relations, by reversing the intersection and union.

\begin{corollary}
\label{Cor:equiv_min_max}
The following holds:
\begin{enumerate}
\item if $\{E^\alpha\}_\alpha$ is closed under countable intersection, then there exists an element $E^{\bar \alpha}$ such that if $\mathcal M^\alpha$ is the complete $\sigma$-algebra generated by $E^\alpha$ according to \eqref{Equa:gen-sigma} then
\begin{equation*}
\forall \alpha, \Omega \in \mathcal M^\alpha \, \exists \Omega' \in \mathcal M^{\bar \alpha} \big( \mu (\Omega \Delta \Omega') = 0 \big);
\end{equation*}
\item if $\{E^\alpha\}_\alpha$ is closed under countable union, then there exists an element $E^{\bar \alpha}$ such that if $\mathcal M^\alpha$ is the complete $\sigma$-algebra generated by $E^\alpha$ according to \eqref{Equa:gen-sigma} then
\begin{equation*}
\forall \alpha, \Omega \in \mathcal M^{\bar \alpha} \, \exists \Omega' \in \mathcal M^{\alpha} \big( \mu (\Omega \Delta \Omega') = 0 \big).
\end{equation*}
\end{enumerate}
\end{corollary}


\subsection{Application to test plans}
\label{Ss:appli_test}

The statement of Proposition \ref{Prop:max_measure_alg} in our setting will be used as follows.

\begin{definition}
\label{Def:eqvui_genset}
Let $\{\Omega_\alpha\}_\alpha$ be a family of subsets of a set $X$. The equivalence relation generated by $\{\Omega_\alpha\}_\alpha$ is the smallest equivalence relation $E$ such that
\begin{equation*}
\bigcup_\alpha \Omega_\alpha \times \Omega_\alpha \subset E.
\end{equation*}
\end{definition}

\begin{lemma}
\label{Lem:min_E}
The equivalence relation generated by $\{\Omega_\alpha\}_\alpha$ is the following:
\begin{equation*}
x E x' \quad \Longleftrightarrow \quad \exists n \in \N, \{\Omega_{\alpha_i}\}_{i=1}^n , x_i,x_{i+1} \in \Omega_{\alpha_i} \big( x_1 = x \ \wedge \ x_{n+1} = x' \big).
\end{equation*}
\end{lemma}

\begin{proof}
Immediate from the definition.
\end{proof}

We apply the above construction to our case, the sets $\Omega_\alpha$ being the range of the curves $\{\gamma \in A_\pi\}$.

\begin{definition}
\label{Def:equiv_rel_ch}
We will say that two points $x,x'$ in $X$ are equivalent according to the set $A_\pi \subset \Gamma$ and we will write $x \gls{EApi} x'$ if they are equivalent according to Definition \ref{Def:eqvui_genset} where the sets $\Omega_\alpha$ are the family of sets $\{\gamma([0,1]), \gamma \in A_\pi\}$. Equivalently, there exists a finite set $\{\gamma_i\}_i \subset A_\pi$ of curves and a path $\gamma \in \Gamma$ such that
\begin{equation*}
\Graph \gamma \subset \bigcup_i \Graph \gamma_i \quad \text{and} \quad \gamma(0) = x, \ \gamma(1) = x'.
\end{equation*}
\end{definition}

For every test plan $\pi \in \adm(\mu)$ consider the family of sets $A_\pi$ satisfying Theorem \ref{theo:neglinters} (i.e. every closed subloop of $A_\pi$ is of integer cicuitation), and let $E^{A_\pi}$, $\mathcal M^{E^{A_\pi}}$ be the corresponding equivalence relation and $\sigma$-algebra, respectively (see \eqref{Equa:gen-sigma}). 

\begin{proposition}
\label{Prop:exi_min_pi}
There exists \gls{barApi} such that $\pi(\bar A_\pi) = 1$ and $\mathcal M^{E^{\bar A_\pi}}$ is maximal. 
\end{proposition}

In particular, for every subset $A_\pi \subset \bar A_\pi$ such that $\pi(A_\pi) = 1$ it holds 
\begin{equation}
\label{Equa:remov_1}
\mathcal M^{E^{\bar A_\pi}} = \mathcal M^{E^{A_\pi}} \quad \text{as measure algebras.}
\end{equation}

We will denote $\mathcal M^{E^{\bar A_\pi}}$ as \gls{calMpi}, and the corresponding equivalence relation by $\gls{Epi} = E^{\bar A_{\bar \pi}}$. We stress that the only unique object is $\mathcal M^\pi$ considered as a complete $\sigma$-subalgebra of $\mathcal M$ or as a measure algebra.

\begin{proof}
We need only to verify that the family $\{E^{A_\pi}\}_{A_\pi}$ is closed under countable intersection: this follows from the implications
\begin{equation*}
\forall i \big( \pi(A_{\pi}^i) = 1 \big) \quad \Rightarrow \quad \pi \bigg( \bigcap_i A^i_\pi \bigg) = 1,
\end{equation*}
and
\begin{equation*}
E^{\cap_i A^i_\pi} \subset \bigcap_i E^{A^i_\pi}. \qedhere
\end{equation*}
\end{proof}

%
%
%

We next apply Proposition \ref{Prop:min_measure_alg} to sequences of test plans.

\begin{proposition}
\label{Prop:perf_test}
There exists $\gls{barpi} \in \adm(\mu)$ for every $\pi \in \adm(\mu)$ it holds
\begin{equation*}
\mathcal M^{\bar \pi} \subset \mathcal M^\pi \quad \text{as measure algebras.}
\end{equation*}
\end{proposition}

\begin{proof}
Consider a sequence of minimal equivalence relations $\{E^{\pi_i}\}_i$, and define the test plan
\begin{equation*}
\hat \pi = \frac{1}{\sum_i \frac{2^{-i}}{C_i}} \sum_i \frac{2^{-i}}{C_i} \pi_i,
\end{equation*}
where $C_i$ is the compressibility constant \eqref{Equa:curves_pi} (as before we normalized the sum of measures in order to be a probability). Let $\bar A_{\hat \pi},\bar A_{\pi_i}$ be a set of curves satisfying Proposition \ref{Prop:exi_min_pi} above for $\hat \pi,\pi_i$ respectively.

If we consider the sets
\begin{equation*}
\hat A_{\pi_i} = \bar A_{\hat \pi} \cap \bar A_{\pi_i},
\end{equation*}
then $\pi_i(\hat A_{\pi_i}) = 1$ because $\pi_i \ll \hat \pi$, and then by maximality of $\mathcal M^{E^{\bar A_{\pi_i}}}$ we conclude that by \eqref{Equa:remov_1} that
\begin{equation*}
\mathcal M^{E^{\bar A_{\pi_i}}} = \mathcal M^{E^{\hat A_{\pi_i}}} = \mathcal M^{\pi_i} \quad \text{as measure algebras.}
\end{equation*}
We thus deduce from $E^{\hat A_{\pi_i}} \subset E^{\bar A_{\hat \pi}}$ that $\mathcal M^{\hat \pi} \subset \mathcal M^{\pi_i}$ for all $i$, which shows that we are in the situation of Point (2) of Corollary \ref{Cor:sigma_min_max}: the minimal measure algebra $\mathcal M^{\bar \pi}$ satisfies the statement.
\end{proof}
%
%
%
%

Let $\bar A_{\bar \pi}$ be a $\sigma$-compact set where $\bar \pi$ is concentrated such that $E^{\bar A_{\bar \pi}}$ generates the maximal measure algebra $\mathcal M^{\bar \pi}$, which by the previous proposition is the smallest among all maximal measure algebras $\mathcal M^\pi$.

Let $\{\Omega_n\}_n$ be a countable family of $E^{\bar A_{\bar \pi}}$-saturated sets generating $\mathcal M^{\bar \pi}$ up to negligible sets, and define the equivalence relation
\begin{equation}
\label{Equa:bar_E}
\gls{barE} = \bigcap_n \big( \Omega_n \times \Omega_n \big) \cup \big( (X \setminus \Omega_n) \times (X \setminus \Omega_n) \big).
\end{equation}
By construction
\begin{equation*}
E^{\bar A_{\bar \pi}} \subset \bar E,
\end{equation*}
and in general the inclusion is strict. The equivalence relation $\bar E$ is made of the atoms of the $\sigma$-algebra generated by the $\{\Omega_n\}_n$, whose completion coincides with $\mathcal M^{\bar \pi}$.

\begin{remark}
\label{Rem:up_to_n_E}
Consider another set $\bar A'_{\bar \pi}$ generating $\mathcal M^{\bar \pi}$, and let $\{\Omega'_n\}_n$ be $E^{\bar A'_{\bar \pi}}$-saturated sets generating $\mathcal M^{\bar \pi}$. Since $\mathcal M^{\bar \pi}$ is unique and it is the completion of the $\sigma$-algebra generated by $\bar E$, then there are $\bar E$-saturated sets $\{\Omega''_n\}_n$  such that
\begin{equation*}
\mu(\Omega''_n \Delta \Omega'_n) = 0.
\end{equation*}
Defining
\begin{equation*}
N = \bigcup_n \Omega''_n \Delta \Omega'_n, \quad \mu(N)=0,
\end{equation*}
we deduce that the equivalence relation obtained by \eqref{Equa:bar_E} using the family $\{\Omega_n'\}_n$ is equal to the one constructed through the sets $\{\Omega''_n\}_n$ up to a $\mu$-negligible set. Indeed the equivalence relation generated by $\{\Omega''_n \setminus N\}_n$ contains $\bar E$ since the sets $\Omega''_n$ are $\bar E$-saturated. Hence
\begin{equation*}
E^{\bar A'_{\bar \pi}} \setminus N \times N \subset \bar E.
\end{equation*}
Reversing the argument we obtain that there exists a $\mu$-negligible set $N'$ such that
\begin{equation*}
\bar E \setminus N' \times N' \subset E^{\bar A'_{\bar \pi}},
\end{equation*}
and then $E^{\bar A'_{\bar \pi}}$ is equal to $\bar E$ up to a $\mu$-negligible set on $X$.
\end{remark}

Let $\pi \in \adm(\mu)$ and $\bar A_\pi$ generating the maximal $\mathcal M^\pi$ as in Proposition \ref{Prop:exi_min_pi}. By minimality of $\mathcal M^{\bar \pi}$, for every $\Omega_n$ there is a set $\Omega^\pi_n \in \mathcal M^\pi$ such that
\begin{equation}
\label{Equa:Omega_n_pi}
\mu \big( \Omega_n \Delta \Omega_n^\pi \big) = 0.
\end{equation}
In particular, following the reasoning of the previous remark, there exists a $\mu$-negligible set $N \subset X$ such that if $\bar E$ is given by \eqref{Equa:bar_E} then
\begin{equation}
\label{Equa:contr_E_bar_E}
E^{\bar A_\pi} \cap (X \setminus N \times X \setminus N) \subset \bar E \cap (X \setminus N \times X \setminus N).
\end{equation}
Since $\gamma \in \bar A_\pi$ takes values in a single class of $E^{\bar A_\pi}$, we conclude that
%
%

\begin{corollary}
\label{Cor:set_good}
Every $\pi \in \adm(\mu)$ is concentrated on a family of trajectories $A_\pi$ such that
\begin{equation*}
\forall \gamma \in A_\pi \ \big( \gamma(s) \bar E \gamma(t) \ \text{up to a $\mathcal L^1$-negligible subset of $[0,L(\gamma)]$} \big).
\end{equation*}
\end{corollary}

\begin{proof}
Since we know that $\gamma(t) E^{\bar A_\pi} \gamma(s)$ for $s,t \in [0,L(\gamma)]$ by construction, then we deduce from \eqref{Equa:contr_E_bar_E} that
\begin{equation*}
\gamma(s) \bar E \gamma(t) \quad \forall s,t \in [0,L(\gamma)] \setminus \gamma^{-1}(N).
\end{equation*}
Being $N$ $\mu$-negligible,
\begin{equation*}
\int \bigg[ \int_0^{L(\gamma)} \ind_N(\gamma(t)) dt \bigg] \pi(d\gamma) = 0,
\end{equation*}
so that we need only to remove the negligible set of curves $\gamma$ such that $\mathcal L^1(\gamma^{-1}(N)) > 0$. 
\end{proof}

\begin{remark}
\label{Rem:N_presence}
The presence of the negligible set $N$ is natural: take for example the space $X = \mathbb S^1$, $\mu = \mathcal L^1$, and the equivalence relation
\begin{equation*}
E = \{(0,0)\} \cup (\mathbb S^1 \setminus \{0\})^2.
\end{equation*}
The measure algebra is the trivial one, but any curve passing through $0$ sees $N = \{0\}$.
\end{remark}

\section{Construction of the function $w$}
\label{S:constr_w}

In the previous section we have constructed an equivalence relation $E^{\bar A_{\bar \pi}}$ for the test plan $\bar \pi$ such that the measure algebra $\mathcal M^{\bar\pi}$ generated by the saturated sets $\mathcal M^{\bar \pi}$ satisfies the following properties:
\begin{enumerate}
\item $\bar A_{\bar \pi}$ satisfies Theorem \ref{theo:neglinters},
\item it is the maximal measure algebra among all measure algebra generated by sets where $\bar \pi$ is concentrated,
\item all others test plans $\pi$ are concentrated on a set of curves $\bar A_\pi$ generating a measure algebra which contains $\mathcal M^{\bar \pi}$.
\end{enumerate}
Nothing is said about the strong consistency of the disintegration w.r.t. $E^{\bar A_{\bar \pi}}$: if
\begin{equation}
\label{Equa:mu_bar_E_dis}
\mu = \int \mu_\alpha \zeta(d\alpha), \quad \zeta = \mu \llcorner_{\mathcal M^{\bar \pi}},
\end{equation}
then in general the conditional probability $\mu_\alpha$ is not concentrated on the corresponding equivalence class $E^{\bar A_{\bar \pi}}_\alpha$: this implies that $\bar E \supseteq E^{\bar A_{\bar \pi}}$.

\begin{assumption}
\label{Ass:disi_cons}
We assume that the minimal disintegration w.r.t. $E^{\bar A_{\bar \pi}}$ is strongly consistent for every $\bar A_{\bar \pi}$ generating $\mathcal M^{\bar \pi}$, i.e.
\begin{equation*}
\mu_\alpha(E^{\bar A_{\bar \pi}}) = 1 \quad \text{for $\zeta$-a.e. $\alpha$}.
\end{equation*}
\end{assumption}

We recall that this is equivalent to each of the following conditions (\cite{BiaCar} or \cite[Section 452]{Fre:measuretheory4}):
\begin{enumerate}
\item the equivalence relation $E^{\bar A_{\bar \pi}}$ is essentially countably generated, and in particular we can assume that $\bar E = E^{\bar A_{\bar \pi}}$ up to a $\mu$-negligible set, where $\bar E$ is given by \eqref{Equa:bar_E};
\item there exists a $\mu$-measurable map $f : X \to [0,1]$ whole level sets are the equivalence classes:
\begin{equation}
\label{Equa:f_determ_equiv}
\big\{ (x,x'):f(x) = f(x') \big\} = E^\pi;
\end{equation}
\end{enumerate}
Moreover, strong consistency implies that, up to a $\mu$-negligible set of equivalence classes of $E^{\bar A_{\bar \pi}}$, there exists a Borel section $S$, i.e. a Borel set in $X$ such that it contains exactly one point of each equivalence class of $E^{\bar A_{\bar \pi}}$.

We collect the previous observations in the following statement. 

\begin{lemma}
\label{Lem:struct_cons_max}
Under Assumption \ref{Ass:disi_cons}, there exists a test plan $\bar \pi$ concentrated on a set of trajectories $\bar A_{\bar \pi}$ satisfying Theorem \ref{theo:neglinters} such that:
\begin{enumerate}
\item the equivalence relation $\bar E = \bar E^{\bar A_{\bar \pi}}$ generated by $\bar A_{\bar \pi}$ admits a Borel section \gls{barS} and has the following equivalent properties (up to a $\mu$-negligible saturated set):
\begin{enumerate}
\item the disintegration \eqref{Equa:mu_bar_E_dis} is consistent,
\item the atoms of the separable $\sigma$-algebra $\mathcal M^{\bar \pi}$ are the equivalence classes of $\bar E$;
\end{enumerate}
\item every two points of an equivalence class are connected by finitely many curves in $\bar A_{\bar \pi}$.
\end{enumerate} 
\end{lemma}


Assumption \ref{Ass:disi_cons} is necessary to prove that the integer circuitation condition of $\Lambda$ implies the existence of the function $w$: indeed if the disintegration is not consistent, then there are counterexamples as the following one.

\begin{example}
\label{Ex:notgoodsol}
Consider the Cantor $1/3$-set
\begin{equation*}
K = [0,1] \setminus \bigcup_{n \in \N} \bigcup_{s \in \{0,2\}^{n-1}} \sum_{i=1}^{n-1} s_i 3^{-i} + 3^{-n} (1,2),
\end{equation*}
and let $V : [0,1] \to [0,1]$ be the Vitali function
\begin{equation*}
V \in C^0([0,1]), \qquad V \bigg( \sum_{i=1}^{n-1} s_i 3^{-i} + 3^{-n} (1,2) \bigg) = \sum_{i=1}^{n-1} s_i 2^{-1-i} + \{2^{-n}\}, \quad s_i \in \{0,2\}.
\end{equation*}
Define for $\alpha \in [0,1] \setminus \Q$
\begin{equation*}
f : K \to K,  \quad f(x) = V^{-1} \circ (\Id +\alpha \mod 1) \circ V(x),
\end{equation*}
which is defined up to the countably many points belonging to the boundary $Z$ of the intervals $\sum_{i=1}^{n-1} s_i 3^{-i} + 3^{-n} (1,2)$, $n \in \N,s_i \in \{0,2\}$. For definiteness we can assume that those points are removed from $K$, in which case the set $K \setminus Z$ is a $G_\delta$-set, i.e. still a Polish space.

It is fairly easy to construct a family of totally disconnected curves $\gamma : [0,1] \times K \to \R^3$ of length $1$ such that
\begin{equation*}
\gamma(0,x) = (0,0,x), \quad \gamma(1,x) = (0,0,f(x)).
\end{equation*}
The MMS $(X,d,\mu)$ is the the union of all these trajectories with the Euclidean distance in $\R^3$, and with the measure $\mathcal L^1 \times \mathcal H^{1/3}$ in the coordinates $(t,x) \in [0,1) \times K$ see Figure \ref{Fig:counterexample}.

\begin{figure}
\centering
\def\svgwidth{\textwidth}
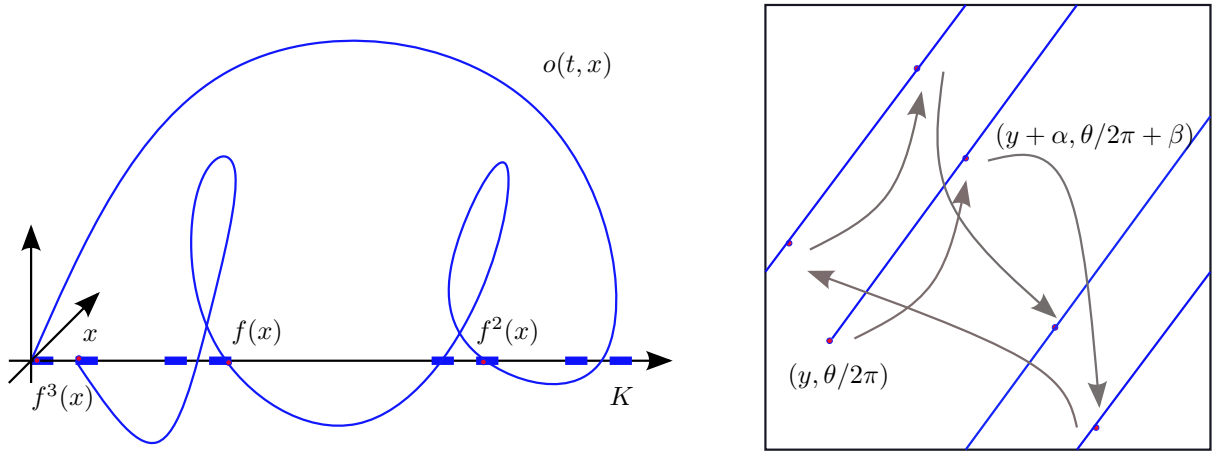
\caption{The Cantor set $K$ together with a piece of the orbit $o$, and the evolution of the point $(y,\theta)$ in the angular coordinates.}
\label{Fig:counterexample}
\end{figure}

Let $o(\cdot,x) : \R \to \R^3$ be the orbit starting from $x \in K$ and defined by
\begin{equation*}
o(t,x) = \gamma(t \mod 1,f^{\lfloor t \rfloor}(x)), \quad \lfloor t \rfloor \ \text{integer part of $t$}.
\end{equation*}
The fact that the segments $\gamma$ are totally disconnected implies that the only admissible plans $\pi$ are concentrated on curves which are subsets of $o(\cdot,x)$, and in particular any measurable function $v : [0,1] \times K \to \R$ gives an integer (or free) circuitation cotangent vector field, being the tangent space one dimensional in each point. 

Setting $v = 2\pi \beta$ constant with $z \cdot (\alpha,\beta) \notin \Z$ for all $z \in \Z^2 \setminus 0$, we obtain that any candidate $w$ such that $\dot w = i w v$ should satisfy
\begin{equation*}
w(0,0,f(x)) = w(0,0,x) e^{2\pi i \beta},
\end{equation*}
which, calling
$$
w'(y) = w(0,0,V^{-1}(y)),
$$
by definition of $f$ gives
\begin{equation*}
w'(y + \alpha \mod 1) = e^{i \theta(y + \alpha \mod 1)} = w'(y) e^{2\pi i \beta} = e^{i \theta(y) + 2\pi i \beta},
\end{equation*}
where the function $\theta : [0,1] \to \mathbb S^1$ is the angle.

If we consider the torus $\mathbb T^2$, then the values of the angles $\theta(y)$ on a single orbit corresponds to the trajectory of the points
\begin{equation*}
n \mapsto \bigg( y + n \alpha \mod 1, \frac{\theta(y)}{2\pi} + n \beta \mod 1 \bigg).
\end{equation*}
We are looking for the invariant graph $\theta = \theta(y)$. Consider the measure $\varpi = \theta_\sharp \mathcal L^1$ and its Fourier transform
\begin{equation*}
\hat \varpi = \sum_{z \in \Z^2} c_z e^{2\pi i z \cdot (y,\theta)}.
\end{equation*}
The invariance implies
\begin{equation*}
\sum_{z \in \Z^2} c_z e^{2\pi i z \cdot (y,\theta)} = \sum_{z \in \Z^2} c_z e^{2\pi i z \cdot (y,\theta) + 2\pi i z \cdot (\alpha,\beta)},
\end{equation*}
i.e.
\begin{equation*}
\forall c_z \not= 0 \ \big( z \cdot (\alpha,\beta) \in \Z^2 \big).
\end{equation*}
This is impossible for $z \not= 0$ by the assumptions on $\alpha,\beta$, hence there are no integrals $w$ to $Dw = i w v$.

We can slightly change the space $X$ and the test plan $\bar \pi$ in order to see that the assumption that the disintegration is strongly consistent for all $\bar A_{\bar \pi}$ is essential. Set
\begin{equation*}
\tilde X = X \cup [0,1] \times \{0\},
\end{equation*}
i.e. we add the segment containing $K$. This segment is negligible for all test plans, since we still keep the measure supported on $X$. To have a set of curves with integer circuitation, just remove a point to every orbit $o(t,x)$, for example assume that
\begin{equation*}
\bar A_{\bar \pi} = \big\{ (t,0,0), t \in [0,1] \big\} \cup \bigcup_{x \in K,n \in \N} \big\{ o(t,x), t \in [0,1/2-2^{-n}] \big\} \cup \big\{ o(-t,x), t \in [0,1/2-2^{-n}] \big\}.
\end{equation*}
It is fairly easy to see that $\bar A_{\bar \pi}$ satisfies Theorem \ref{theo:neglinters}, and the equivalence class is $\tilde X \setminus \cup_x \{o(1/2,x)\}$. However removing the negligible curve $(t,0,0)$ the measure algebra $\mathcal M$ is not the minimal one $\mathcal M^{\bar \pi}$.

The fact that the last possibility is ruled out by Assumption \ref{Ass:disi_cons} is essential in the proof of Theorem \ref{Theo:w_final}.
%
%
%
%
%
%
\end{example}

%

We now prove the existence of a function $w$ using the curves of $\bar A_{\bar \pi}$.

\begin{proposition}
\label{Prop:exist_w}
There exists a Borel function $\check w$ such that for all sub-curves $\gamma$ of $\bar A_{\bar \pi}$ it holds
\begin{equation*}
\check w(\gamma(1)) = \check w(\gamma(0)) e^{i \int_0^1 \Lambda_\gamma(t) \dot{|\gamma|}(t) dt}.
\end{equation*}
\end{proposition}

\begin{proof}
Let $\bar S$ be a Borel section of $\bar E$, and let \gls{GammaApiS} be the family of subcurves of $\bar A_{\bar \pi}$ passing through $\bar S$. We will assume that $\bar S$ has a point in each equivalence class, by removing a $\mu$-negligible set of equivalence classes. 

For every curve $\gamma \in \Gamma_{\bar A_{\bar \pi},\bar S}$ let
\begin{equation*}
s(\gamma) = \max \big\{ t : \gamma(t) \in \bar S \big\}.
\end{equation*}
Notice that by the assumption that $\bar A_{\bar \pi}$ generates $\bar E$, every point $x \in X$ is connected to the unique point $\bar S \cap Ex$ by a subcurve $\gamma \in \Gamma_{\bar A_{\bar \pi},\bar S}$. Being $\{ (t,\gamma), \gamma(t) \in \bar S\}$ a Borel set, it follows that $s$ is Borel.


Define the function 
\begin{equation}
\label{Equa:w_def}
\check w : [0,1] \times \Gamma_{\bar A_{\bar \pi},\bar S} \to L^\infty((0,1)), \quad \check w(t,\gamma) = e^{i \int_{s(\gamma)}^t \Lambda(\gamma)(\tau) \dot{|\gamma|}(\tau) d\tau}.
\end{equation}
It is easy to see that this function is Borel, being $s(\gamma)$ and $\Lambda$ Borel. Moreover, from the integer circuitation assumption on $\bar A_{\bar \pi}$, if $\gamma(t) = \gamma'(t')$, $\gamma,\gamma' \in \Gamma_{\bar A_{\bar \pi},\bar S}$, then
\begin{equation*}
\frac{\check w(\gamma'(t'))}{\check w(\gamma(t))} = e^{i \big( \int_{s(\gamma')}^{t'} \Lambda(\gamma')(\tau) \dot{|\gamma|}(\tau) d\tau + \int_t^{s(\gamma)} \Lambda(\gamma)(\tau) \dot{|\gamma|}(\tau) d\tau \big)} = 1,
\end{equation*}
because the curve
\begin{equation*}
\gamma(\tau) = \begin{cases}
R_{s(\gamma'),t'}(\gamma')(2 \tau) & \tau \in [0,1/2], \\
R_{t,s(\gamma)}(\gamma)(2(1-\tau)) & \tau \in (1/2,1],
\end{cases}
\end{equation*}
is a sub-loop of $\bar A_{\bar \pi}$. This shows that $\check w(t,\gamma)$ is the graph of a Borel function $\check w$ defined on $X$.
%
%
%
%
%
%

Finally, let $\gamma \in \Gamma_{\bar A_{\bar \pi}}$ be a subcurve of $\bar A_{\bar \pi}$. Let $\gamma_0,\gamma_1 \in \Gamma_{\bar A_{\bar \pi},\bar S}$ be subcurves connecting $\gamma(0),\gamma(1)$ to a point $\bar S$, respectively: such curves exists because $\gamma$ takes values in a single equivalence class, and each point in an equivalence class is connected to $\bar S$ by a subcurve. The subloop made by these 3 curves has integer circuitation, and then
\begin{equation}
\label{Equa:triang_circu}
\begin{split}
\check w(\gamma(1)) &= \check w(\gamma(0)) e^{- i \int_1^{-s(\gamma_0)} \Lambda(\gamma_0) |\dot \gamma_0| \mathcal L^1 + \int_{s(\gamma_1)}^1 \Lambda(\gamma_1) |\dot \gamma_1| \mathcal L^1} \\
&= \check w(\gamma(0)) e^{i \int_0^1 \Lambda(\gamma) \dot{|\gamma|} \mathcal L^1} e^{- i \int_0^1 \Lambda(\gamma) \dot{|\gamma|} \mathcal L^1 - i \int_1^{-s(\gamma_0)} \Lambda(\gamma_0) |\dot \gamma_0| \mathcal L^1 + \int_{s(\gamma_1)}^1 \Lambda(\gamma_1) |\dot \gamma_1| \mathcal L^1} \\
&= \check w(\gamma(0)) e^{i \int_0^1 \Lambda(\gamma) \dot{|\gamma|} \mathcal L^1}.
\end{split}
\end{equation}
where we have used that the curve obtained by joining $\gamma,\gamma_0,\gamma_1$ is a subloop of $\bar A_{\bar \pi}$.
\end{proof}

%
%
%
%

The next proposition shows the importance that $\mathcal M^{\bar \pi}$ is the maximal measure algebra generated by $\bar \pi$.

\begin{proposition}
\label{Prop:zero_Lambda}
If $w$ satisfies
\begin{equation}
\label{Equa:equa_cic_1}
\frac{d}{dt} w(\gamma(t)) = 0, \quad \pi\text{-a.e.} \ \gamma,
\end{equation}
for every test plan $\pi$, then up to a $\mu$-negligible set $w$ is constant in each equivalence class.
\end{proposition}

Observe that in this prof we do not need the assumption on the strong consistency of the disintegration, Assumption \ref{Ass:disi_cons}: this is due to the fact that the function $w$ is $\mu$-measurable.

\begin{proof}
Consider the test plan $\bar \pi$ generating the minimal equivalence relation $\bar E$. 

Assumption \eqref{Equa:equa_cic_1} implies that $t \mapsto w(\gamma(t))$ is equal to a constant $h(\gamma)$ for $\mathcal L^1$-a.e. $t \in [0,1]$ for $\bar \pi$-a.e. $\gamma$. Define the set
\begin{equation*}
W = \Big\{ (\gamma,\gamma') : \gamma([0,1]) \cap \gamma'([0,1]) \not= \emptyset, h(\gamma) \not= h(\gamma') \Big\} \subset \Gamma \times \Gamma.
\end{equation*}
The same analysis of Step 1 of Theorem \ref{theo:neglinters} implies that there is a $\bar \pi$-negligible set $N \subset \Gamma$ such that
\begin{equation*}
\forall \gamma,\gamma' \in \Gamma \setminus N \Big( \gamma([0,1]) \cap \gamma'([0,1]) \ \Rightarrow \ h(\gamma) = h(\gamma') \Big).
\end{equation*}
Indeed, for every measure $0 < \xi \in \Pi^\leq(\bar \pi,\bar \pi)$ concentrated on $W$ the test plan obtained by \eqref{Equa:mergiloop} would not satisfy the assumption of the proposition. 

Hence, $t \mapsto w(\gamma(t))$ is constant in each subcurve $\gamma$ of $\Gamma \setminus N$ for $\mathcal L^1$-a.e. $t$. In particular, it is constant in each equivalence class of the equivalence relation $E^{\bar A_{\bar \pi} \setminus N}$ generated by $\bar A_{\bar \pi} \setminus N$.

This equivalence relation $E^{\bar A_{\bar \pi} \setminus N}$ is smaller than $\bar E$, but by the maximality of $\mathcal M^{\bar E}$ it generates the same measure algebra $\mathcal M^{\bar E} = \mathcal M^{\bar \pi}$. If $w$ is not constant in each equivalence class of $\bar E$ up to a $\mu$-negligible set $N' \subset X$, then it follows that taking $f$ as in \eqref{Equa:f_determ_equiv} for the equivalence relation $\bar E$ and defining
\begin{equation*}
f'(x) = \big( f(x),w(x)),
\end{equation*}
we obtain a Borel map such that its equivalence relation $\bar E' = (f',f')^{-1}(\Id)$ is smaller than $\bar E$ and the measure algebra $\mathcal M^{\bar E'}$ strictly contains $\mathcal M^{\bar \pi}$, which contradicts the maximality of the latter.
%
%
\end{proof}

The last result gives the general formula for every solution to the differential equation $Dw = iwv$ or more precisely
\begin{equation*}
\frac{d}{dt} w \circ \gamma = i w(\gamma) \Lambda_\gamma(t) \dot{|\gamma|}(t),
\end{equation*}
rewritten in \eqref{Equa:equa_cic} below.

\begin{theorem}
\label{Theo:w_final}
Given a function $\gls{omega} \in L^\infty(\zeta,\mathbb S^1)$ ($\zeta$ given by \eqref{Equa:mu_bar_E_dis}), there exists a function $w \in L^\infty(\mu,\mathbb S^1)$ solving the equation for every test plan $\pi$
\begin{equation}
\label{Equa:equa_cic}
w(\gamma(t)) = w(\gamma(s)) e^{i \int_s^t \Lambda_\gamma \dot{|\gamma|} \mathcal L^1}, \quad \pi\text{-a.e.} \ \gamma, \ s,t \in [0,1] \setminus N_\gamma, \mathcal L^1(N_\gamma) = 0, 
\end{equation}
and such that
$$
w(x) = \omega(y) \check w(x) \quad \text{for $\zeta$-a.e. $y \in \bar S \cap \bar Ex$},
$$
where $\check w$ is the function constructed in Proposition \ref{Prop:exist_w}.

Viceversa, if $w$ is a solution to the equation \eqref{Equa:equa_cic} above, then the ratio $w/\check w$ is constant on each equivalence class $\bar E_y$ for $\zeta$-a.e. $y$ and defines a function $\omega \in L^\infty(\zeta,\mathbb S^1)$ on the quotient space $X/\bar E$.
\end{theorem}

In particular the map
$$
L^\infty(\zeta,\mathbb S^1) \ni \omega(y) \longleftrightarrow \omega(\bar Ex) \check w(x) \in L^\infty(\mu,\mathbb S^1)
$$
is bijective on the set of solutions.

\begin{proof}
Given a function $\omega \in L^\infty(\zeta,\mathbb S^1)$, the following extension of formula \eqref{Equa:w_def} constructs the candidate function for the first part of the statement:
\begin{equation}
\label{Equa:w_def_1}
w(x;\omega) = \omega(\bar Ex) e^{i \int_{s(\gamma)}^1 \Lambda(\gamma) \dot{|\gamma|} \mathcal L^1} = \omega(Ex) \check w,
\end{equation}
where 
$\gamma \in \Gamma_{\bar A_{\bar \pi},\bar S}$ is a subcurve of $\bar A_{\bar \pi}$ connecting $y$ to $x$. 

We need to prove that the integral formula \eqref{Equa:equa_cic} holds for all $\pi$. Let $\pi$ be a test plan: by removing a negligible set of curves, we can assume that $(\pi + \bar \pi)/2$ is concentrated on a set of integer circuitation curves, together with its subloops (Corollary \ref{Cor:count_m}). Being the equivalence relation $\bar E$ maximal in the sense of Proposition \ref{Prop:perf_test}, we can apply Corollary \ref{Cor:set_good}: the curves $\gamma$ on which $\pi$ is concentrated are such that
\begin{enumerate}
\item $\gamma(t)$ belongs to the same equivalence class of $\bar E$ up to a $\mathcal L^1$-negligible subset of $[0,1]$; 
\item the function $w(\gamma(t);\omega)$ given by formula \eqref{Equa:w_def_1} is defined $\mathcal L^1$-a.e.
\end{enumerate}
Since the disintegration is strongly consistent, for every $\gamma$ as above there is a subcurve curve $\gamma'$ of the set where $\bar \pi$ is concentrated connecting $\gamma(s)$ to $\gamma(t)$ (by Point (1) above and Assumption \ref{Ass:disi_cons}). Being $\pi + \bar \pi$ of integer circuitation, we deduce as in \eqref{Equa:triang_circu} that \eqref{Equa:equa_cic} holds.

%

Finally we show that every solution $w$ is computed in this way, i.e. the second part of the statement. By construction, along $\bar \pi$-a.e. $\gamma$
\begin{equation*}
\frac{d}{dt} \frac{w(\gamma(t))}{\check w(\gamma(t))} = \frac{w(\gamma(t))}{\check w(\gamma(t))} \big( \Lambda(\gamma)(t) \dot{|\gamma|}(t) - \Lambda(\gamma)(t) \dot{|\gamma|}(t) \big) = 0,
\end{equation*}
so that we can apply Proposition \ref{Prop:zero_Lambda}.
\end{proof}

\section{The $W^{1,2}$-case}
\label{S:W12}

In this section we apply the above analysis to the case $\mu = |f|^2 m/\|f\|_2^2$, where $m$ is now the ambient measure on $X$ and $f \in W^{1,2}(X,m)$. We still need some assumptions on the structure of the MMS $(X,d,m)$, as explained in the introduction.

\begin{assumption}
\label{Assu:local_test}
For $m$-a.e. point $x \in X$ there exist two sequence of numbers $r_n,R_n$, $r_n \leq R_n$, converging to $0$ and a test plan $\pi_n \in \mathcal P(\Gamma)$ such that for some constant $C_x$
\begin{equation*}
\int_0^1 e(t)_\sharp (\dot{|\gamma|} \pi_n) dt \leq \frac{C_x}{m(B_{r_n}(x))} m \llcorner_{B_{R_n}(x)}, \quad L(\gamma) \leq 2R_n, \ m(B_{R_n}(x)) \leq C_x m(B_{r_n}(x)),
\end{equation*}
and
\begin{equation*}
(e(0),e(1))_\sharp \pi_n = \frac{1}{m(B_{r_n}(x))^2} m \times m \llcorner_{B_{r_n}(x) \times B_{r_n}(x)}.
\end{equation*}


%
\end{assumption}

Up to the restriction that the curves $\gamma$ used by $\pi_n$ are localized inside the ball $B_{R_n}(x)$, the test plans $\pi_n$ are the Democratic plans of \cite{lottvillani:weakcur}.

\begin{example}
\label{Ex:torus_Riemann}
The above condition holds for example for $[0,1]^d$, $d \in \N$, with $m = \prod_i \mathcal L^1 \llcorner_{[0,1]}$: indeed one considers the image measure
\begin{equation*}
((1-t)x + ty)_\sharp ( m \times m )
\end{equation*}
which is equivalent to $m$ by coarea formula, being
\begin{equation*}
\sqrt{\sum_i (\det M_i)^2} = \sqrt{((1-t)^2+t^2)^d} \simeq 1,
\end{equation*}
where $M_i$ are the $d \times d$-minors of the matrix $[(1-t) \Id \ t \Id]$. In particular, it holds in a Riemann manifold by a slight variation of the above computation. 

More generally, it holds in non-branching metric measure space under the Measure Contraction Property $MCP(K,N)$ \cite[Definition 2.1]{otha:contr}, just by adapting the proof of \cite[Theorem 3.1]{lottvillani:weakcur}.

In the case of branching spaces, the same condition holds if (locally, being all computations done in a small ball) 
\begin{equation}
\label{Equa:meas_balls_contr}
\frac{1}{C} \leq \frac{m(B_r(x))}{r^N} \leq C
\end{equation}
and for every $x \in X$, $0 < r < 1$, there is a transport plan $\pi_{x,r} \in \mathrm{Geo}(X)$ concentrated on geodesics such that
\begin{equation}
\label{Equa:MCP+1}
e(0)_\sharp \pi_{x,r} = \delta_x, \ e(1)_\sharp \pi_{x,r} = \frac{m \llcorner_{B_r(x)}}{m(B_r(x))}, \quad e(t)_\sharp \pi_{x,r} \leq \frac{C}{t^N} \frac{m \llcorner_{B_{rt}(x)}}{m(B_r(x))}. 
\end{equation}
In particular, it holds for branching $MCP(K,N)$-spaces with the Hausdorff measure $\mathcal H^N$ as ambient measure $m$ \cite{rajala:interpo}.

To prove this fact, consider any ball $B_{\bar r}(\bar x)$, $0 < \bar r < 1/2$, and define
\begin{equation*}
\bar \pi = \fint_{B_{\bar r}(\bar x)} \tilde \pi_{y,\bar x,\bar r} m(dy), \quad \tilde \pi_{y,\bar x,\bar r} = \frac{m(B_{2\bar r}(y))}{m(B_{\bar r}(\bar x))} \pi_{y,2\bar r} \llcorner_{\{\gamma(1) \in B_{\bar r}(\bar x)\}}.
\end{equation*}
Being $\pi_{x,\bar r} \in \mathcal P(\mathrm{Geo}(X))$, i.e. it is concentrated on the set of geodesic, it follows that 
$L(\gamma) \leq 2 \bar r$, and thus $e(t)_\sharp \pi_{y,\bar x,\bar r}$ is concentrated on $B_{2\bar r}(\bar x)$ for $y \in B_{\bar r}(\bar x)$. The measure contraction property (the last estimate of \eqref{Equa:MCP+1}) gives that $m(B_{2\bar r}(\bar x)) \leq C 2^{N} m(B_{\bar r}(\bar x))$. Moreover by construction
$$
(e(0),e(1))_\sharp m = \frac{1}{m(B_{\bar r}(\bar x))^2} m \times m \llcorner_{B_{\bar r}(\bar x) \times B_{\bar r}(\bar x)}.
$$
Finally, consider a ball $B_r(x)$ and compute
\begin{equation*}
\begin{split}
e(t)_\sharp \bar \pi(B_r(x)) &= \fint e(t)_\sharp \pi_{y,\bar x,\bar r}(B_r(x)) m(dy) \leq \min \bigg\{  1, \frac{C m(B_r(x))}{t^N m(B_{\bar r}(\bar x))} \bigg\} \frac{m(B_{r+2t \bar r}(x))}{m(B_{\bar r}(\bar x))},
\end{split}
\end{equation*}
where we have used the last estimate in \eqref{Equa:MCP+1} and the simple observation
\begin{equation*}
e(t)_\sharp \bar \pi_{y,\bar x,\bar r}(B_r(x)) \leq \begin{cases}
1 & y \in B_{r+2t \bar r}(x), \\
0 & \text{otherwise}.
\end{cases}
\end{equation*}
Using again the measure contraction property one obtains 
\begin{equation*}
m(B_{r+t \bar r}(x)) \leq C \bigg( 1 + \frac{t \bar r}{r} \bigg)^N m(B_r(x)),
\end{equation*}
and then for $t \geq r/\bar r$ and using \eqref{Equa:meas_balls_contr}
\begin{equation*}
\begin{split}
\frac{C m(B_r(x))}{t^N m(B_{\bar r}(\bar x))} \frac{m(B_{r + 2t \bar r}(x))}{m(B_{\bar r}(\bar x))} &\leq C^2 \bigg( \frac{1}{t} + \frac{2 \bar r}{r} \bigg)^N \frac{m(B_r(x))^2}{m(B_{\bar r}(\bar x))^2} \leq \frac{C^2 (3\bar r)^N}{r^N} \frac{m(B_r(x))^2}{m(B_{\bar r}(\bar x))^2} 
\\
&\leq C^4 \frac{m(B_r(x))}{m(B_{\bar r}(\bar x))}. 
\end{split}
\end{equation*}
For $t \leq r/\bar r$ we have $m(B_{r+2t \bar r}(x)) \leq C 4^N m(B_r(x))$, and thus we obtain
\begin{equation*}
e(t)_\sharp \bar \pi(B_r(x)) \leq \frac{C''}{m(B_{\bar r}(x))} m(B_r(x))
\end{equation*}
for some constant $C''$. This shows that $e(t)_\sharp \bar \pi \leq C'' m/m(B_{\bar r}(\bar x))$.
\end{example}

%
%
%

\begin{lemma}
\label{Lem:one_class}
There exists a test plan $\pi$ satisfying Assumption \ref{Assu:local_test} and such that the equivalence relation $\bar E$ given by \eqref{Equa:bar_E} is made of countably many equivalence classes of positive measure.
\end{lemma}

\begin{proof}
Take a point $x$ as in Assumption \ref{Assu:local_test}, fix $n$ and let $A_{\pi_n}$ be a set of curves where $\pi_n$ is concentrated: by assumption $m$-a.e. point $x_1 \in B_{r_n}(x)$ is connected by a curve to $m$-a.e. point $x_2 \in B_{r_n}(x)$. By the transitivity of the equivalence relation, all these points are in the same equivalence class, i.e. $B_{r_n}(x)$ is in a single equivalence class up to a $m$-negligible set. It is clear that by removing a negligible set of trajectories we changes the equivalence class only in a $m$-negligible set.

Let $\{\Omega_n\}_n = \{B_{r_n}(x_n)\}_n$ be the family of balls as above covering $m$-a.e. $x \in X$. Let $\pi_{nn'}$ be a test plan (if it exists) such that
\begin{equation*}
e(0)_\sharp \pi_{nn'} \ll m \llcorner_{B_{r_n}(x_n)}, \quad e(1)_\sharp \pi_{nn'} \ll m \llcorner_{B_{r_{n'}}(x_{n'})}.
\end{equation*}
The test plan of the statement is
\begin{equation*}
\bar \pi = \sum_n c_n \pi_n + \sum_{nn'} c_{nn'} \pi_{nn'},
\end{equation*}
with $c_n,c_{nn'}$ chosen as in \eqref{Equa:countabl_pi} to have a probability measure.
%
%
%
%
%
\end{proof}

%

We next consider a function $f \in W^{1,2}(X,m)$ 
and define probability measure
\begin{equation*}
\mu = \frac{1}{\|f\|_2^2} |f|^2 m.
\end{equation*}
We assume that we are given the cotangent vector field $v \in L^2(T^*X)$ on the space $(X,\mu)$, and we denote with $\Lambda_\gamma(t)$ the corresponding function as in \eqref{Equa:plan_ct}.

\begin{proposition}
\label{Prop:pos_measure}
There exists a test plan $\bar \pi$ (w.r.t. the measure $\mu$) satisfying Assumption \ref{Ass:disi_cons} such that its equivalence classes are of positive measure.
\end{proposition}

\begin{proof}
As in the previous proof, it is enough to prove the statement locally by showing that nearby to a Lebesgue point of $|f|,|Df| > 0$ there exists a plan $\bar \pi$ generating an equivalence class of positive measure essentially invariant w.r.t. the choice of the sets $A_{\bar \pi}$.

Let $\bar x \in X$ be a Lebesgue point for $|f|$ and $|Df|$ (being $|Df|$ the upper gradient), i.e. for definiteness
\begin{equation}
\label{Equa:Lebesgue_f_Df}
\fint_{B_{\bar r}(\bar x)} |f - 1|^2 m, \fint_{B_{\bar r}(\bar x)} ||Df| - 1|^2 m < \epsilon^2.
\end{equation}
Consider the ball $B_{r_n}(\bar x)$, with $R_n \leq \bar r$, and the test plan $\pi_n$ defined in Assumption \ref{Assu:local_test}: we have
%
%
\begin{equation*}
\begin{split}
\int \bigg[ \int_0^1 |Df|(\gamma(t)) \dot{|\gamma|}(t) dt \bigg] \pi_n(d\gamma) &= \int L(\gamma) \pi_n(d\gamma) + \int \bigg[ \int_0^1 \big[ |Df|(\gamma(t)) - 1 \big] \dot{|\gamma|}(t) dt \bigg] \pi_n(d\gamma) \\
\big[ L(\gamma) \leq 2R_n, \ e(t)_\sharp \pi_n \leq C_{\bar x} m/m(B_{r_n}(\bar x)) \big] \quad &\leq 2R_n + C_{\bar x} \int_{B_{R_n}(\bar x)} ||Df|-1| \frac{m}{m(B_{r_n}(\bar x))} \\
&\leq 2R_n + C_{\bar x} \epsilon \frac{m(B_{R_n}(\bar x))}{m(B_{r_n}(\bar x))} \leq 2 R_n + C_{\bar x}^2 \epsilon.
\end{split}
\end{equation*}
%
%
In particular, using Chebyshev inequality, the $\pi$-measure of the set $\mathcal W$ of trajectories for which 
\begin{equation*}
\int_0^1 |Df|(\gamma(t)) \dot{|\gamma|}(t) dt > \frac{1}{2}
\end{equation*}
is bounded by
\begin{equation*}
\pi(\mathcal W) = \pi \bigg( \bigg\{ \gamma : \int_0^1 |Df|(\gamma(t)) \dot{|\gamma|}(t) dt > \frac{1}{2} \bigg\} \bigg) < 4 R_n + 2 C_{\bar x}^2 \epsilon < \frac{1}{64}
\end{equation*}
for $n \gg 1$.
%
%
%
%

By the disintegration w.r.t. the map $e(0) : \gamma \mapsto \gamma(0)$,
\begin{equation*}
\pi_n = \fint_{B_{r_n}(\bar x)} \pi_{n,x} m(dx), 
\end{equation*}
it follows that
\begin{equation*}
\frac{1}{m(B_{r_n}(\bar x))} m \bigg( \bigg\{ x : \pi_{n,x}(\mathcal W) \geq \frac{1}{8} \bigg\} \bigg) < \frac{1}{8}. 
\end{equation*}
For all other points, there exists a family of geodesics of $\pi_{n,x}$-measure greater than $7/8$ such that
\begin{equation}
\label{Equa:bound_df}
\int_0^1 ||Df(\gamma(t))| - 1|^2 \dot{|\gamma|} dt \leq \frac{1}{2}.
\end{equation}
By Assumption \ref{Assu:local_test},
\begin{equation*}
e(1)_\sharp \check \pi_{n,x} = \frac{m}{m(B_{r_n}(\bar x))},
\end{equation*}
so that the end points of the previous trajectories cover a subset of $B_\delta(\bar x)$ with $m$-measure larger than $7 m(B_{r_n}(\bar x))/8$.

The set of $x$ such that $|f(x) - 1|> 1/4$ has measure smaller than $16\epsilon^2 m(B_{r_n}(\bar x))$, so that there is a set of positive $m$-measure such that $|f(x) - 1| > 1/4$ and the trajectories starting from $x$ verifying \eqref{Equa:bound_df} have end points in a set with measure $> 7m(B_{r_n}(\bar x))/8$. In particular the trajectories starting from each of these points must intersect, and since
\begin{equation*}
|f(\gamma(t))| \geq |f(\gamma(0))| - \int_0^1 |Df(\gamma(s))| \dot{|\gamma|}(s) ds \geq \frac{1}{4},
\end{equation*}
then they live in an equivalence class of positive measure.
%
%
%
%
%
%
%
%
\end{proof}

\begin{corollary}
\label{Cor:disint}
The disintegration w.r.t. the equivalence classes of $E^{\bar A_{\bar \pi}}$ for the measure $\mu$ is strongly consistent for every $\bar A_\pi$, hence that there exists a solution $w$.
\end{corollary}

This corollary, together with Example \ref{Ex:torus_Riemann}, proves the existence part of \eqref{Equa:cicuit_dP} of Theorem \ref{Theo:gen_case}.

\begin{proof}
We just observe that the disintegration with equivalence classes of positive measure is certainly strongly consistent, corresponding to a countable partition of the ambient space up to a negligible set.
\end{proof}

Finally, we consider the case $f \in W^{1,2}(X,m)$ and $v \in L^2(T^*X)$ cotangent vector field in $(X,\mu)$. The above computations can be applied to the equivalence classes of $\mu = f^2 m$, and then by Theorem \ref{Theo:w_final} we can construct a function $w \in L^\infty(\mu,\mathbb S^1)$ such that
\begin{equation*}
w(\gamma(1)) - w(\gamma(0)) = \int_0^1 \Lambda(\gamma) \dot{|\gamma|} \mathcal L^1, \quad \pi\text{-a.e. } \gamma, \ \pi \ \text{test plan of $(X,\mu)$}.
\end{equation*}

Let now $\pi'$ be a test plan for $(X,m)$: we can assume that $\pi$ is concentrated on a family of curves $A_\pi$ such that $f \circ \gamma \in H^1([0,L(\gamma)])$ for all $\gamma \in A_\pi$. In particular, $f \circ \gamma$ is continuous and for every $n \in \N$ define
\begin{equation}
\label{Equa:f_gam_-1_inter}
\big\{ |f \circ \gamma| > 2^{-n} \big\} = \bigcup_k \big( t^-_{k,n}(\gamma),t^+_{k,n}(\gamma) \big).
\end{equation}
The next lemma assures that the dependence w.r.t. $k,\gamma$ of $t^\pm_{k,n}$ is Borel.

\begin{lemma}
\label{Lem:Borel_H_1}
There exists a family of countable maps of Borel regularity
\begin{equation*}
\gamma \mapsto t^-_{k,n}(\gamma),t^+_{k,n}(\gamma)
\end{equation*}
such that \eqref{Equa:f_gam_-1_inter} holds.
\end{lemma}

\begin{proof}
Being $(t,\gamma) \mapsto e(t,\gamma) = \gamma(t)$ continuous, by selecting a Borel representative of $f$ we obtain that
\begin{equation*}
\tilde f(t,\gamma) = \big[ |f(\gamma(t))| - 2^{-n} \big]^+
\end{equation*}
is Borel, and by the choice of $A_\pi$ it is continuous w.r.t. $t$ for every $\gamma \in \Gamma$ by setting $\tilde f(t,\gamma) = 0$ if $\gamma \notin A_\pi$. We will prove the statement for the function $\tilde f$.

Let $s_\ell \in [0,1] \cap \Q$ be a countable dense sequence of times, set
\begin{equation*}
A_{\pi,k} = \{\gamma \in A_\pi: \tilde f(s_k,\gamma) > 0\} 
\end{equation*}
and define the function $\tilde t^+_k : A_{\pi,k} \to [0,1]$ as the function whose graph is
\begin{equation*}
\big\{ \gamma : \tilde t^+_k(\gamma) \geq t \big\} = \bigcap_{s_k \leq s_\ell \leq t} A_{\pi,\ell}.
\end{equation*}
In the same way the function $\tilde t^-_k : A_{\pi,k} \to [0,1]$ is given by 
\begin{equation*}
\big\{ \gamma : \tilde t^-_k(\gamma) \leq t \big\} = \bigcap_{t \leq s_\ell \leq s_k} A_{\pi,\ell}.
\end{equation*}
Begin $A_{\pi,k}$ Borel, the functions $\tilde t^\pm_k$ are Borel with domain $A_{\pi,k}$. It is clear that with this procedure we have
\begin{equation*}
\{\tilde f > 0\} = \bigcup_\gamma \bigcup_k \big( t^-_k(\gamma),t^+_k(\gamma) \big).
\end{equation*}

The only step left is to have that every interval of $\{\tilde f(\cdot,\gamma) > 0\}$ is counted once. Set $t_{1,n}^\pm(\gamma) = \tilde t_1^\pm(\gamma)$, and in general
\begin{equation*}
t^\pm_{k,n}(\gamma) = \tilde t^\pm_k(\gamma) \quad \text{if} \ k = \min \big\{ j : s_j \in (\tilde t^-_k(\gamma),\tilde t^+_k(\gamma)) \big\}.
\end{equation*}
Being the last set a Borel set, the functions $t^\pm_k$ are again Borel. 
\end{proof}

By defining
\begin{equation*}
\gamma_{k,n} = \gamma \llcorner_{(t^-_{k,n}(\gamma),t^+_{k,n}(\gamma))}, \quad \pi_{k,n} = (\gamma_{k,n})_\sharp \pi,
\end{equation*}
we have that $\pi_{k,n}$ is a test plan for the measure $\mu = f^2 m/\|f\|_2^2$, with compressibility constant $2^{2n} C/\|f\|_2^2$, $C$ being the compressibility constant of $\pi$. Along $\pi_{k,n}$-a.e. curve, by elementary 1d computations we have that
\begin{equation}
\label{Equa:chain_fw}
\frac{d}{dt} (fw)(\gamma(t)) = \frac{d}{dt}f(\gamma(t)) w(\gamma(t)) + f(\gamma(t)) \frac{d}{dt} w(\gamma(t)),
\end{equation}
so that, recalling $w \in \mathbb S^1$, for $\pi_{n,k}$-a.e. $\gamma$ we deduce
\begin{equation}
\label{Equa:fw_W12}
\bigg| \frac{d}{dt} (fw) \circ \gamma(t) \bigg| \leq |Df|(\gamma(t)) + |f(\gamma(t))| |Dw(\gamma(t))| = |Df|(\gamma(t)) + |f(\gamma(t))| |v(\gamma(t))|.
\end{equation}
The last function belongs to $L^2(X,m)$, because $f \in W^{1,2}(X,m)$ and $|Dw| \in L^2(X,\mu)$, i.e.
\begin{equation*}
\int |Dw|^2 |f|^2 \mu < \infty.
\end{equation*}
By letting $n \to \infty$ and extending $t \mapsto (fw)(\gamma(t))$ as a continuous function, we deduce that \eqref{Equa:chain_fw}, \eqref{Equa:fw_W12} holds for $\pi$-a.e. $\gamma$, and since trivially $|fw| = |f| \in L^2(X,m)$ we conclude with the following proposition.


\begin{proposition}
\label{Prop}
Under the assumptions $f \in W^{1,2}(X,m)$, $v \in L^2(X,|f|^2 m)$, there exists a solution $w$ such that $f w \in W^{1,2}(X,m)$ and
\begin{equation}
\label{Equa:Dfw_expl}
D(fw) = wDf + f Dw.
\end{equation}
Moreover any other solution $w'$ satisfies
\begin{equation*}
\frac{w(x)}{w'(x)} = \omega(Ex), 
\end{equation*}
where $E$ is the equivalence relation constructed in Section \ref{S:constr_min_eq} in the space $(X,\mu)$ with $\mu = f^2m/\|f\|_2^2$.
\end{proposition}

In particular there are at most $(\mathbb S^1)^\N$-different solutions. Together with Proposition \ref{Prop:pos_measure}, the above statement completes the proof of Theorem \ref{Theo:gen_case}.

\begin{proof}
Because of Theorem \ref{Theo:w_final}, we are left to prove \eqref{Equa:Dfw_expl}: this follows from \eqref{Equa:chain_fw} and the observation that every cotangent vector field is uniquely defined by its behavior on test plans as in \eqref{Equa:simple_plan_ct}.
\end{proof}

\newpage 
\appendix
\printglossaries

\newpage 
\bibliographystyle{alpha}
\bibliography{biblio}

\end{document}